\documentclass[a4paper, 11pt]{amsart}

\usepackage{verbatim}
\usepackage{amssymb}
\usepackage{textcomp}
\usepackage[all]{xy}

\theoremstyle{plain}
\newtheorem*{thm}{Theorem}

\newtheorem{theorem}{Theorem}[section]
\newtheorem{lemma}[theorem]{Lemma}
\newtheorem{corollary}[theorem]{Corollary}
\newtheorem{proposition}[theorem]{Proposition}

\theoremstyle{definition}
\newtheorem{definition}[theorem]{Definition}
\newtheorem{example}[theorem]{Example}
\newtheorem{question}[theorem]{Question}
\newtheorem*{ques}{Question}

\newtheorem{remark}[theorem]{Remark}

\theoremstyle{remark}
\newtheorem*{claim}{Claim}
\newtheorem*{acknowledgements}{Acknowledgements}

\numberwithin{figure}{section}           
\numberwithin{equation}{section}

\DeclareMathOperator{\aut}{Aut}

\DeclareMathOperator{\Th}{Th}

\DeclareMathOperator{\ch}{char}

\DeclareMathOperator{\sym}{Sym}
\DeclareMathOperator{\alt}{Alt}
\DeclareMathOperator{\supp}{supp}

\DeclareMathOperator{\cw}{cw}

\DeclareMathOperator{\htt}{ht}

\DeclareMathOperator{\codim}{codim}

\DeclareMathOperator{\GL}{GL}
\DeclareMathOperator{\PGL}{PGL}
\DeclareMathOperator{\SL}{SL}
\DeclareMathOperator{\PSL}{PSL}

\DeclareMathOperator{\SO}{SO}

\DeclareMathOperator{\tp}{tp}

\DeclareMathOperator{\diag}{diag}

\DeclareMathOperator{\im}{Im}

\DeclareMathOperator{\rank}{rank}

\DeclareMathOperator{\st}{st}

\newcommand{\PP}{{\mathbb P}}
\newcommand{\N}{{\mathbb N}}
\newcommand{\Z}{{\mathbb Z}}
\newcommand{\Q}{{\mathbb Q}}
\newcommand{\R}{{\mathbb R}}

\newcommand{\F}{{\mathbb F}}
\newcommand{\GG}{{\mathbb G}}

\newcommand{\M}{{\mathcal M}}

\newcommand{\CC}{{\mathcal C}}
\newcommand{\W}{{\mathcal W}}
\newcommand{\G}{{\mathcal G}}

\newcommand{\U}{{\mathcal U}}
\newcommand{\Gi}{{\mathcal G}}
\newcommand{\RR}{{\mathcal R}}
\newcommand{\HH}{{\mathcal D}}

\newcommand{\Res}{{\mathrm {R}}}

\newcommand{\dotcup}{\ensuremath{\mathaccent\cdot\cup}} 

\begin{document}

\title{Absolute connectedness and classical groups}
\author{Jakub Gismatullin}
\thanks{The author was supported by the Polish Government MNiSW grants N N201 384134, N N201 545938, by the project SFB 478, Fellowship START of the Foundation for Polish Science and by the Marie Curie Intra-European Fellowship MODGROUP no. PIEF-GA-2009-254123.}
\date{\today}

\address{School of Mathematics, University of Leeds, Woodhouse Lane, Leeds, LS2 9JT, UK}
\address{and}
\address{Instytut Matematyczny Uniwersytetu Wroc{\l}awskiego, pl. Grunwaldzki 2/4, 50-384 Wroc{\l}aw, Poland}
\address{and}
\address{Institute of Mathematics of the Polish Academy of Sciences, ul. \'Sniadeckich 8, 00-956 Warsaw, Poland}
\email{gismat@math.uni.wroc.pl, www.math.uni.wroc.pl/\~{}gismat}


\begin{abstract}
We introduce and develop the model-theoretic notions of \emph{absolute connectedness} and \emph{type-absolute connectedness} for groups. We prove that groups of rational points of split semisimple linear groups (that is, \emph{Chevalley groups}) over arbitrary infinite fields are absolutely connected and characterize connected Lie groups which are type-absolutely connected. We prove that the class of type-absolutely connected group is exactly the class of discretely topologized groups with trivial \emph{Bohr compactification}, that is the class of \emph{minimally almost periodic groups}. As an application we generalize some results of Conversano-Pillay and construct a group $G$ where $G^{00}/G^{\infty}$ is far from being abelian.
\end{abstract}

\maketitle


\section*{Introduction}

Suppose $(G,\cdot,\ldots)$ is an arbitrary infinite group with some first order structure. In model theory, after passing to a sufficiently saturated extension $G^*$ of $G$, we consider several kinds of \emph{model-theoretic connected components} of $G$. Let $A\subset G^*$ be a small set of parameters and define (cf. \cite{sh1,gis,NIP,modcon,NIP2,conv_pillay,bcg}) the following: 
\begin{itemize}
\item ${G^*}^0_A$ (the connected component of $G$ over $A$) is the intersection of all $A$-definable subgroups of ${G^*}$ which have finite index in ${G^*}$,
\item ${G^*}^{00}_A$ (the type-connected component of $G$ over $A$) is the smallest subgroup of `bounded index' in ${G^*}$ (bounded relative to $|{G^*}|$), that is, type-definable over $A$,
\item ${G^*}^{\infty}_A$ (the $\infty$-connected  component of $G$) is the smallest subgroup of bounded index in ${G^*}$, that is, $A$-invariant (invariant under the automorphisms of ${G^*}$ fixing $A$ pointwise).
\end{itemize} 
In the literature the component ${G^*}^{\infty}_A$ is sometimes denoted by ${G^*}^{000}_A$ \cite{NIP,NIP2,conv_pillay}. The groups ${G^*}^0_A$, ${G^*}^{00}_A$ and ${G^*}^{\infty}_A$ are important in the study of groups from a model-theoretic point of view. In general, the quotients ${G^*}/{G^*}^{\infty}_A$, ${G^*}/{G^*}^{00}_A$ and ${G^*}/{G^*}^{0}_A$ with the \emph{logic topology} \cite{pill} are topological groups which are \emph{invariants} of the theory $\Th(G)$. That is, they do not depend on the choice of saturated extension $G^*$ (cf. \cite[Proposition 3.3(3)]{gis}). We have ${G^*}^{\infty}_A\subseteq {G^*}^{00}_A\subseteq {G^*}^0_A$, also ${G^*}/{G^*}^{0}_A$ is a \emph{profinite} group, ${G^*}/{G^*}^{00}_A$ is a \emph{compact} Hausdorff group and ${G^*}/{G^*}^{\infty}_A$ is a \emph{quasi-compact} topological group, that is, compact but not necessary Hausdorff (cf. \cite[Proposition 3.5(1)]{gis} where we used the notation $G^*_L$ for ${G^*}^{\infty}$). Hence, if we equip $G$ with the discrete topology, then for all $A\subseteq G$ the mappings $G \to {G^*}/{G^*}^{00}_A$ and $G \to {G^*}/{G^*}^{0}_A$ are \emph{compactifications} of $G$ in the sense of Section \ref{sec:rad}.

We focus on the type-connected component ${G^*}^{00}_A$ and the $\infty$-connected component ${G^*}^{\infty}_A$. The goals of the paper are: to study the notions of \emph{absolute connectedness} and \emph{type-absolute connectedness}, to give examples of (type-)absolutely connected groups and to derive some applications. We say that a group $G$, without any additional first order structure, is \emph{absolutely connected} [resp. \emph{type-absolutely connected}] if for \emph{all} possible first order expansions $(G,\cdot,\ldots)$ of $G$, and for all sufficiently saturated extensions $G^*$ and small $A\subset G^*$, we have ${G^*}^{\infty}_A = {G^*}$ [resp. ${G^*}^{00}_A = {G^*}$, for every $A \subseteq G$] (see \ref{def:abscon}, \ref{def:typeabscon}).

The main novelty in this approach to the model-theoretic connected components derives from considering groups equipped with an \emph{arbitrary} first order structure. It extends the previous settings (NIP, $o$-minimal). In many examples these connected components are determined by the group structure alone, that is, by some group-theoretic properties. For example, every boundedly simple group (Section \ref{sec:abscon}) is absolutely connected.

There are not many known examples where ${G^*}^{\infty}_{\emptyset} \neq {G^*}^{00}_{\emptyset}$. The first such example has been found in \cite{conv_pillay}, which was then generalized in \cite{comcentr}. Every example of this kind gives a new non-$G$-compact theory \cite[Section 3]{gis}, that is, a theory in which some Lascar strong type is not type-definable. The first example of a non-$G$-compact theory was given in \cite{ziegler}, and there are not many examples known. For all examples in \cite{conv_pillay,comcentr}, the group ${G^*}^{00}_{\emptyset} / {G^*}^{\infty}_{\emptyset}$ is abelian. The structure of ${G^*}^{00}_{\emptyset} / {G^*}^{\infty}_{\emptyset}$ is not well understood in general. In particular the following question has been formulated in \cite{conv_pillay}.

\begin{ques}
Can ${G^*}^{00}_{\emptyset} / {G^*}^{\infty}_{\emptyset}$ be non-abelian?
\end{ques}

We answer this question affirmatively by giving a general construction where ${G^*}^{00}_{\emptyset} / {G^*}^{\infty}_{\emptyset}$ is very far from being abelian (Theorem \ref{thm:nonab}). In our construction we use the notion of a \emph{generalized quasimorphism} (\ref{def:quasim}) and a method of recovering a compact Hausdorff group from a dense subgroup (\ref{prop:recov}).

As an immediate consequence of (\ref{prop:recov}) we also obtain a classification of all compactifications of a Hausdorff topological group as mappings $G\to G^*/H$, for some type-definable bounded index subgroup $H$ of $G^*$. Similar classification has been proved by Robinson and Hirschfeld (\cite{hir}) but in the context of the \emph{enlargement} $^*G$ from non standard analysis. As a corollary of (\ref{prop:recov}) we obtain the following characterization of the class of type-absolutely connected groups (\ref{thm:tcvN}):
\begin{quote}
The class of type-absolutely connected groups coincides with the class of discretely topologized minimally almost periodic groups, that is, the class of groups without any nontrivial compactifications.
\end{quote}
This gives a direct link between notions from topology and model theory. Using this observation and \cite{shtern2} we can characterize all connected Lie groups which are type-absolutely connected (\ref{rem:shtern}). In particular, the topological universal cover $\widetilde{\SL_2(\R)}$ of $\SL_2(\R)$ is type absolutely connected (this generalizes results from \cite{conv_pillay}). Also we introduce and study, in Section \ref{sec:rad}, the \emph{type-connected radical} of a group (\ref{def:tcrad}).

Other goal is to prove Theorem \ref{thm:che2} below concerning the structure of groups of rational points of split semisimple linear groups (also called Chevalley groups). We find the bound 12 for absolute connectedness. Recall that a linear group defined over a field $k$ is called \emph{split over $k$} or \emph{$k$-split} if some maximal torus $T$ in $G$ is split over $k$, that is, $T$ is isomorphic over $k$ to a direct product of copies of the multiplicative groups $\GG_m$ \cite[18.6]{borel}. Over an algebraically closed field, every semisimple group is split.

\begin{thm}{\bf \ref{thm:che2}}
Let $k$ be an arbitrary infinite field and $G$ be a $k$-split, semisimple linear algebraic group $G$ defined over $k$. The derived subgroup $[G(k),G(k)]$ of $G(k)$ is 12-absolutely connected. Moreover $[G(k),G(k)]$ is a $\emptyset$-definable subgroup of $G(k)$ in the pure group language.
\end{thm}

The proof of (\ref{thm:che2}) goes through the Gauss decomposition with prescribed semisimple part from \cite{cheg}. We first prove absolute connectedness of simply connected semisimple split groups (Theorem \ref{thm:che}) and then, using universal $k$-coverings, establish the general case. If $G$ is not simply connected, then the group $[G(k),G(k)]$ might be a proper Zariski dense subgroup of $G(k)$. For example if $G=\PGL_n$, then $[G(k),G(k)] = \PSL_n(k)$ is a proper subgroup for some fields. Likewise the result of \cite{plat-died} implies that if $G$ is a non-simply connected semisimple $k$-split $k$-group and $k$ is a finitely generated field, then $[G(k),G(k)]$ is a proper subgroup of $G(k)$.

Some results of Sections \ref{sec:abscon} and \ref{sec:other} of this paper appeared in the author's Ph.D. thesis.

\section{Basic notation and prerequisites} 

In this section we establish notation and recall some basic facts, mainly from \cite{gis, modcon}.

\subsection{Group theory} \label{sec:gr}

For a group $G$, elements $a,b\in G$ and subsets $A,B\subseteq G$ we use the following notation: $A\cdot B = \{ab : a\in A, b \in B\}$, $a^b = b^{-1}ab$, $A^B = \bigcup_{a\in A,b\in B}a^b$, $A^n = \underbrace{A\cdot \ldots \cdot A}_{n\text{ times}}$, $A^{\leq n} = \bigcup_{i\leq n} A^i$ and $[a,b] = a^{-1}b^{-1}ab$. The subset $X\subseteq G$ is called \emph{normal} if $X^G = X$. The derived series is defined by $G^{(1)} = [G,G]$, $G^{(n+1)}=\left[G^{(n)},G^{(n)}\right]$, $G^{(\lambda)} = \bigcap_{\alpha<\lambda}G^{(\alpha)}$. A group $G$ is called \emph{perfect} if $G=[G,G]$. The \emph{perfect core} of $G$ is the largest perfect subgroup of $G$. It is also the intersection of all elements from the transfinite derived series of $G$. The \emph{commutator length} of an element $g \in [G,G]$ is the minimal number of commutators sufficient to express $g$ as their product. The \emph{commutator width} $\cw^G(X)$ of a subset $X\subseteq G$ is the maximum of the commutator lengths of elements from $X$ or the sign $\infty$ if the maximum does not exist or $X\not\subseteq [G,G]$. The number $\cw^G(G)$ is denoted by $\cw(G)$. By $Z(G)$ we denote the center and by $e$ the neutral element of $G$.

\subsection{Model theory}

We assume that the reader is familiar with basic notions of model theory. The model-theoretic background can be found in \cite{marker}. We usually work in a \emph{monster model} or \emph{sufficiently saturated model}, that is, $\kappa$-saturated and $\kappa$-strongly homogeneous model $\M^*$, for sufficiently big $\kappa$ (over some language $L$). By $A$ we always denote some \emph{small}, that is, $|A|<\kappa$, set of parameters from $\M^*$, and by $L(A)$ the set of all $L$-formulas with parameters from $A$ (also called \emph{$A$-formulas}). \emph{$A$-definable} means definable by an $A$-formula; \emph{$A$-type-definable} means definable by a conjunction of a family of $A$-formulas; \emph{$A$-invariant} means invariant under the group $\aut\left(\M^*/A\right)$ of automorphisms of $\M^*$ stabilizing $A$ pointwise. An equivalence relation $E$ on $\M^*$ is called \emph{bounded} if $E$ has fewer than $\kappa$ many equivalence classes. It is well known that if $E$ is type-definable over $A$ or $A$-invariant, then either the number of classes of $E$ is at most $2^{|L(A)|}$ or at least $\kappa$.

Usually we will deal with groups with some first order structure, but sometimes we consider groups only in a group language without any extra structure. In the latter case we say that the group is \emph{pure}.

\subsection{Boundedness and thickness} \label{sec:bt}
A formula $\varphi(\overline{x},\overline{y})\in L(A)$, with tuples $\overline{x},\overline{y}$ ($|\overline{x}|=|\overline{y}|$) of free variables, is called \emph{thick} if $\varphi$ is symmetric and for some $n\in\N$, for every $n$-sequence $(\overline{a_i})_{i<n}$ from $\M^*$ (where we do not require $\overline{a}_0,\dots, \overline{a}_{n-1}$ to be pairwise distinct) there exist $i<j<n$ such that $\varphi(\overline{a}_i,\overline{a}_j)$ holds in $\M^*$ (\cite[Section 3]{ziegler}). By $\Theta_A(\overline{x},\overline{y})$ we denote the conjunction of all thick formulas over $A$. The relation $\Theta_A$ is type-definable over $A$, but is not necessarily transitive. It has the following known properties (\cite[Lemmas 6,7]{ziegler}, \cite[1.11, 1.12]{pillay}):
\begin{itemize}
\item[(\ref{sec:bt}.1)] If $\Theta_A(\overline{a},\overline{b})$, then there is a small model $M'\prec M$ containing $A$, such that $\overline{a}\underset{M'}{\equiv}\overline{b}$, that is, $\tp(\overline{a}/M)=\tp(\overline{b}/M)$.
\item[(\ref{sec:bt}.2)] If for some small model $M'\prec M$ containing $A$, $\overline{a}\underset{M'}{\equiv}\overline{b}$ holds, then ${\Theta_A}^2(\overline{a},\overline{b})$.
\end{itemize}

\subsection{Model-theoretic connected components} \label{sec:modcon}

Unless otherwise stated, we assume in this subsection that $G$ is a group equipped with extra first order structure which is a monster model. Let $H$ be an $A$-invariant subgroup of $G$. The index $[G:H]$ is either at most $\leq 2^{|L(A)|}$ or at least $\geq\kappa$. In the first case we say that $H$ has \emph{bounded} index in $G$. The following subgroups are called the \emph{model-theoretic connected components} of $G$.
\begin{itemize}
\item $G^{0}_{A}   =  \bigcap\{ H < G : H \text{ is } A \text{-definable and } [G:H] \text{ is finite}\}$
\item $G^{00}_{A}  = \bigcap\{ H < G : H \text{ is type-definable over } A \text{ and } [G:H] \text{ is bounded} \}$
\item $G^{\infty}_{A} =  \bigcap\{ H < G : H \text{ is } A \text{-invariant and } [G:H] \text{ is bounded}\}$
\end{itemize}

If for every small set of parameters $A\subset G$, $G^{\infty}_A = G^{\infty}_{\emptyset}$, then we say that $G^{\infty}$ \emph{exists} and define it as $G^{\infty}_{\emptyset}$. Similarly we define existence of $G^{00}$ and $G^{0}$. By \cite{modcon}, the groups $G^{\infty}_{A} \subseteq G^{00}_{A} \subseteq G^{0}_{A}$ are normal subgroups of $G$ of bounded index. Following the notation from \cite[Section 1]{modcon}, define $X_{\Theta_A}=\{a^{-1}b : a,b\in G, \Theta_{A}(a,b)\}$. Note that $G^{\infty}_{A}$ is generated by $X_{\Theta_A}$ \cite[Lemma 2.2(2)]{modcon}. In \cite[Lemma 3.3]{modcon} we gave another description of $X_{\Theta_A}$. The key idea is the notion of a \emph{thick subset} of a group, based on the definition of thick formula. 

\begin{definition}[{\cite[Definition 3.1]{modcon}}] \label{def:thick}
A subset $P$ of an arbitrary group $G$ (not necessarily sufficiently saturated) is called \emph{$n$-thick}, where $n\in\N$, if it is symmetric (that is, $P = P^{-1}$) and for every $n$-sequence $g_1,\ldots,g_n$ from $G$, there are $1\leq i<j\leq n$ such that $g_i^{-1}g_j \in P$. We say that $P$ is \emph{thick} if $P$ is $n$-thick for some $n\in\N$ (the condition $g_i^{-1}g_j \in P$ is equivalent to requiring $g_ig_j^{-1} \in P$, as one can take $g_1^{-1},\ldots,g_n^{-1}$ instead of $g_1,\ldots,g_n$).
\end{definition}

Suppose $G$ and $H$ are groups (not necessarily sufficiently saturated), $f \colon G \to H$ is an epimorphism, $P, Q \subseteq G$, $S\subseteq H$ and $n,m\in\N$. By $R(n,m)$ we denote the Ramsey number. 

\begin{lemma} \label{lem:thick}
\begin{enumerate}
\item If $P$ is $n$-thick and $Q$ is $m$-thick, then $P\cap Q$ is $R(n,m)$-thick.
\item The preimage $f^{-1}(S)$ is $n$-thick if and only if $S$ is $n$-thick. 
\item If $P$ is $n$-thick, then $f(P)$ is $n$-thick.
\item If $H<G$ and $P$ is $n$-thick in $G$, then $H\cap P$ is $n$-thick in $H$.
\end{enumerate}
\end{lemma} 
\begin{proof} $(1)$ Assume that $X\cap Y$ is not $R(n,m)$-thick and take the sequence $(a_i)_{i<R(n,m)}$ witnessing this. Consider the complete graph on $R(n,m)$ vertices with the following coloring: $\{i,j\}$ is black if and only if ${a_i}^{-1}a_j\not\in X$, otherwise $\{i,j\}$ is white. By the Ramsey theorem either there is a clique of size $n$ with all black edges, so $X$ is not $n$-thick, or a clique of size $m$ with white edges, so $Y$ is not $m$-thick. $(2)$, $(3)$ and $(4)$ are immediate.
\end{proof}

\begin{lemma} \label{lem:thicks}
The group $G^{\infty}_A$ is generated by the intersection of all $A$-definable thick subsets of $G$. More generally, suppose $n\in\N$. Then
\begin{enumerate}
\item $X_{\Theta_A}^n = \bigcap \left\{P^n : P\subseteq G \text{ is $A$-definable and thick}\right\}$,
\item $X_{\Theta_A}^G = \bigcap \left\{P^G : P\subseteq G \text{ is $A$-definable and thick}\right\}$,
\item $X_{\Theta_A}^G \subseteq X_{\Theta_A}^4$.
\end{enumerate}
\end{lemma}
\begin{proof} $(1)$ is \cite[Lemma 3.3]{modcon}. $(2)$ $\subseteq$ is clear. For $\supseteq$, take $a$ from the intersection. By $(1)$ for $n=1$ and Lemma \ref{lem:thick}$(1)$, the type $p(x,y) = \{a = x^y,\  x \in P : P\subseteq G \text{ is $A$-definable and thick} \}$ is consistent, so $a\in X_{\Theta_A}^G$. $(3)$ By (\ref{sec:bt}.1) and (\ref{sec:bt}.2) we have \[X_{\Theta_A} \subseteq \bigcup_{M\prec G \text{ small and }A\subseteq M} X_{\underset{M}{\equiv}} \subseteq X_{\Theta_A}^2, \tag{\textreferencemark}\] where $X_{\underset{M}{\equiv}}=\left\{a^{-1}b : a,b\in G, a\underset{M}{\equiv}b\right\}$.  Moreover $\left(X_{\underset{M}{\equiv}}\right)^{G} \subseteq X_{\underset{M}{\equiv}}^2$ holds. Indeed, if $y\in\left(X_{\underset{M}{\equiv}}\right)^{G}$, then there are $a,x \in G$ and $h\in \aut(G/M)$ such that \[y= (a^{-1}h(a))^x = (ax)^{-1}h(a)x =  ((ax)^{-1}h(ax))  (h(x)^{-1}x)\in X_{\underset{M}{\equiv}}^2.\] Hence, by (\textreferencemark), we get $X_{\Theta_A}^{G} \subseteq X_{\Theta_A}^4$.
\end{proof}

\begin{corollary} \label{cor:4}
Let $G$ be a group (not necessarily sufficiently saturated). If $P\subseteq G$ is thick, then there exists a thick subset $Q\subseteq P^4$ which is normal and $\emptyset$-definable in the structure $(G,\cdot,P)$, where $P$ is a predicate.
\end{corollary}
\begin{proof}
Consider $\G = (G,\cdot, P)$ and let $\G^* = (G^*,\cdot, P^*)$ be a sufficiently saturated extension. By Lemma \ref{lem:thicks}, $X_{\Theta_{\emptyset}}^{G^*} \subseteq X_{\Theta_{\emptyset}}^4 \subseteq {P^*}^4$, and by compactness we can find a thick and $\emptyset$-definable $Q\subseteq {G^*}$, with $Q^{G^*} \subseteq {P^*}^4$.
\end{proof}

The following lemma is well known (\cite[3.5]{modcon}, \cite[2.4]{bcg}).

\begin{lemma}\label{lem:00}
The component $G^{00}_A$ can be written as $\bigcap_{i\in I}P_i$ for some family $\{P_i : i\in I\}$ of thick and $A$-definable subsets of $G$ such that 
\[  \text{for every $i\in I$, there is $j\in I$ with $P_j\cdot P_j\subseteq P_i$} \tag{\textdagger}.\]
\end{lemma}

\section{Absolutely connected groups}
\label{sec:abscon}

In the rest of the paper the symbol $(G,\cdot,\ldots)$ denotes a group $G$ possibly equipped with extra first order structure $(\ldots)$. For example, $G$ might be a definable group in some structure $\M$. In this case there is a natural induced structure on $G$ from $\M$ (cf. \cite[Section 1.3]{marker}). We usually do not assume that $G$ is sufficiently saturated. By $G^*$ we always denote a sufficiently saturated elementary extension of $(G,\cdot,\ldots)$.

The aim of the present section is to introduce the notion of (definable) absolute connectedness (Definition \ref{def:abscon}), and give some basic results on it. In short, $G$ is \emph{absolutely connected} if ${G^*}^{\infty}_A = G^*$, for every small $A$, working in any sufficiently saturated extension $G^*$ of any arbitrary expansion of $G$. We introduce also an auxiliary subclass $\W$ of the class of absolutely connected groups (Definition \ref{def:wsim}). 

\begin{proposition} \label{prop:bas0}
The following are equivalent
\begin{enumerate}
\item ${G^*}^{\infty}$ exists and $G^* = {G^*}^{\infty}$,
\item there exists $n\in\N$ such that for every parameter-definable thick subset $P\subseteq G$, $P^n = G$.
\end{enumerate}
\end{proposition}

Note that since every thick set $P$ contains the neutral element $P^n = P^{\leq n}$.

\begin{proof} $(2)\Rightarrow(1)$ It is enough to show that for every small set $A\subset G^*$, $G^* = X_{\Theta_A}^n$. Let $P^* = \varphi(G^*,\overline{a})$ be $k$-thick and $\overline{a} \subset A$. Since in $G$ it is true that ``for every $\varphi$-definable (with parameters) and $k$-thick $P\subseteq G$, $P^n = G$'', the same is true in $G^*$, that is, $\varphi(G^*,\overline{a})^n = G^*$. Hence by Lemma \ref{lem:thicks}$(1)$, $G^* = X_{\Theta_A}^n={G^*}^{\infty}_A$.

$(1)\Rightarrow(2)$ Suppose, on the contrary, that for every $n\in\N$ there is $P_n\subset G^*$, definable over some finite $A_n\subset G^*$ and thick, with $P_n^n\neq G^*$. Let $A=\bigcup_{n\in\N}A_n$. Then clearly, for every $n\in\N$, we have $X_{\Theta_A}^n\subseteq {P_n}^n \neq G^*$. By compactness ${G^*}^{\infty}_A\neq G^*$, which is impossible.
\end{proof}

In the next proposition we consider group $G$ having the property $G^* = {G^*}^{\infty}$ in all first order expansions.

\begin{proposition} \label{prop:bas}
Let $G$ be a non-trivial group. The following conditions are equivalent.
\begin{enumerate}
\item[(1)] There exists $n\in\N$ such that for every thick subset $P\subseteq G$ (not necessarily definable), $P^n = G$.
\item[(2)] $G$ is infinite and if $G^*$ is a sufficiently saturated extension of any first order expansion of $G$, then ${G^*}^{\infty}$ exists and $G^* = {G^*}^{\infty}$.
\end{enumerate}
Furthermore if $(1)$ holds, then ${G^*}^{\infty} = X_{\Theta_A}^n$ for every small $A\subset G^*$ and arbitrary sufficiently saturated extension $G^*$ of $G$.
\end{proposition}
\begin{proof} $(1) \Rightarrow (2)$ $G$ is infinite, because otherwise the thick subset $\{e\}$ of $G$ does not fulfill the condition from $(1)$. The rest of $(2)$ follows easily from Proposition \ref{prop:bas0}.

$(2) \Rightarrow (1)$ Suppose, contrary to our claim, that for every $n\in\N$ there is a thick subset $P_n$ of $G$ with $P_n^n\neq G$. Expand the structure of $G$ by predicates for all $P_n$. If $G^*$ is a sufficiently saturated extension of $G$, then clearly for every $n\in\N$ we have $X_{\Theta_{\emptyset}}^n\subseteq {P^*_n}^n \neq G^*$, so ${G^*}^{\infty}_\emptyset\neq G^*$, contrary to $(2)$.
\end{proof}

In \cite[Section 3]{gis} we made a link between model-theoretic connected components and strong types. Strong types are fundamental objects in model theory and correspond to orbits on $G^*$ of some canonical subgroups of $\aut(G^*)$. In \cite{newelski} Newelski considers the \emph{diameters} of Lascar strong types. Motivated by his idea we introduce below (\ref{def:abscon}) the notion of \emph{$n$-(definable) absolutely connected group}, for $n\in\N$. It turns out that $G$ is $n$-definable absolutely connected if and only if the Lascar strong type of a sort $X$ in a certain structure $(G,\circ,X)$ has diameter at most $n$ \cite[Section 3]{gis}.

\begin{definition} \label{def:abscon}
\begin{enumerate}
\item A group $G$ is called \emph{$n$-absolutely connected} or \emph{$n$-ac} if it satisfies the condition $(1)$ from Proposition \ref{prop:bas}, that is, for every thick subset $P\subseteq G$, we have $P^n = G$.
\item We say that a group $(G,\cdot,\ldots)$ with some first order structure is \emph{$n$-definably absolutely connected}, if for every definable (with parameters) thick $P\subseteq G$, we have $P^n = G$.
\item A group is \emph{[definably] absolutely connected}, if it is $n$-[definably] absolutely connected, for some $n\in\N$.
\end{enumerate}
\end{definition}

It is clear that every elementary extension of an absolutely connected group is definably absolutely connected.

A group $G$ is \emph{$n$-boundedly simple} if for every $g\in G\setminus Z(G)$ and $h\in G$, the element $h$ is the product of $n$ or fewer conjugates of $g^{\pm1}$. Obviously, every boundedly simple group is absolutely connected. In order to give other examples of absolutely connected groups, we introduce auxiliary classes $\W_n$, $n\in\N$ of groups. Each $\W_n$ is a proper subclass of the class of absolutely connected groups (Theorem \ref{thm:conn}, Proposition \ref{prop:prop}).

For $n\in\N$ define \[\Gi_n(G) = \left\{g\in G: \left( g^G \cup {g^{-1}}^G \right)^{\leq n} = G\right\}.\] If $G$ is $n$-boundedly simple, then $\Gi_n(G) = G\setminus Z(G)$. In Definition \ref{def:wsim} below we require the set $\Gi_n(G)$ to be big in a weaker sense.

\begin{definition} \label{def:wsim}
We say that a group $G$ is in class $\W_n$ if $G \setminus \Gi_n(G)$ is not thick; that is, for every $k$, there is a sequence $(g_i)_{i<k}$ in $G$ such that $g_i^{-1}g_j\in \Gi_n(G)$ for all $i<j<k$. We also define $\W=\bigcup_{n\in\N}\W_n$.
\end{definition}

In (\ref{prop:ext}) and (\ref{lem:zam}) below we collect basic properties of absolutely connected groups and groups from $\W$. We give two lemmas for use in the proofs.

\begin{definition} \label{def:generic}
We say that a subset $P \subseteq G$ is \emph{$m$-generic} in $G$ if there are elements $g_1,\ldots,g_m$ from $G$, such that $\bigcup_{1\leq i\leq m}P\cdot g_i = G$.
\end{definition}
 
\begin{lemma}[{\cite[Lemma 3.6, Lemma 3.2(4)]{modcon}}] \label{lem:generic}
\begin{enumerate}
\item If $P$ is a $m$-generic subset of $G$, $1 \in P$ and $P = P^{-1}$, then $P^{3m-2}$ is a subgroup of $G$ of index at most $m$.
\item If $P$ is $m$-generic, then $P^{-1}P$ is $(m + 1)$-thick. If $P$ is $m$-thick, then $P$ is $(m-1)$-generic.
\end{enumerate}
\end{lemma}

\begin{lemma} \label{lem:iwa}
Let $G$ be a perfect group with $\cw(G) = n<\infty$, and $A \subseteq G$ be a normal and symmetric subset. If $B < G$ is a solvable subgroup of derived length $m$ and $G = A \cdot B$, then $G=A^{(4n)^m}$.
\end{lemma}
\begin{proof} Let $C$ be the set of all commutators. By the assumption $G = C^n$. We use the following commutator identity \[[a_1b_1,a_2b_2] = \left(a_1^{-1}\right)^{b_1}\left(a_2^{-1}a_1\right)^{b_1b_2}{a_2}^{{b_2}^{b_1}}\cdot [b_1,b_2]. \tag{\textasteriskcentered}\] It is enough to prove $G = A^{\left(4n\right)^k}\cdot B^{(k)}$ for every $k\in\N$. The case $k=0$ follows from the assumption. For the induction step, note that (\textasteriskcentered) implies that $C\subseteq A^{4\left(4n\right)^k}\cdot \left[B^{(k)},B^{(k)}\right]$ and then $G = A^{{\left(4n\right)}^{k+1}}\cdot B^{(k+1)}$, since $G = C^n$ and $A$ is a normal subset of $G$.
\end{proof}

\begin{proposition} \label{prop:ext}
Let $f\colon G \to H$ be an epimorphism of groups.
\begin{enumerate}
\item[(1)] If $G$ is $n$-ac, then $H$ is also $n$-ac.
\item[(2)] If $H$ is $n_1$-ac and $\ker(f)$ is $n_2$-ac, then $G$ is $(n_1+n_2)$-ac.
\item[(3)] If $G = \bigcup_{i\in I}G_i$ and each $G_i$ is $n$-ac, then $G$ is $n$-ac.
\item[(4)] If $H$ is $n$-ac, $G$ has no subgroups of finite index and $\ker(f)$ is finite of cardinality $m$, then $G$ is $n(3m-2)$-ac.
\item[(5)] If $H$ is $n$-ac, $G$ is perfect with $\cw(G) = r <\infty$ and $\ker(f)$ is solvable of derived length $m$, then $G$ is $4n\left(4r\right)^m$-ac.
\end{enumerate}
\end{proposition}
\begin{proof} Proofs of $(1)-(3)$ are standard.
%
%
%
%
$(4)$ Let $P\subseteq G$ be thick. Then $f(P)^n=H$, so $G = P^n\cdot\ker(f)$ and $P^{n}$ is right $m$-generic. By Lemma \ref{lem:generic}$(1)$, $G = P^{n(3m-2)}$.
$(5)$ Take a thick subset $P\subseteq G$. By (\ref{cor:4}), the set $P^4$ contains a normal thick subset $Q$. Then as in $(4)$, $G = Q^{n}\cdot\ker(f)$. By Lemma \ref{lem:iwa}, $G = Q^{n\left(4r\right)^m} = P^{4n\left(4r\right)^m}.$
\end{proof}

As in (\ref{prop:ext}), we can prove that the class $\W$ is closed under certain operations.

\begin{lemma} \label{lem:zam}
\begin{enumerate}
\item For each $n\in\N$, the class $\W_n$ is elementary in the pure language of groups and is closed under taking homomorphic images, direct sums and arbitrary direct products.
\item Let $f\colon G \to H$ be an epimorphism of groups.
\begin{enumerate}
\item If $H\in\W_n$, $G$ has no subgroups of finite index and $\ker(f)$ is finite of cardinality $m$, then $G\in\W_{n(3m-2)}$.
\item If $H\in\W_n$, $G$ is perfect with commutator width $\cw(G) = r <\infty$ and $\ker(f)$ is solvable of derived length $m$, then $G\in\W_{n\left(4r\right)^m}$.
\end{enumerate}
\end{enumerate}
\end{lemma}
\begin{proof}
$(1)$ This follows from the following remarks: if $f\colon G \to H$ is an epimorphism, then $f\left(\Gi_n(G)\right) \subseteq \Gi_n(H)$, $\Gi_n\left(\bigoplus_{i\in I} G_i\right) = \bigoplus_{i\in I} \Gi_n\left(G_i\right)$ and $\Gi_n\left(\prod_{i\in I} G_i\right) = $ $\prod_{i\in I} \Gi_n\left(G_i\right)$.

$(2)$ The proof is identical to the proof of Proposition \ref{prop:ext}$(4)$ and $(5)$.
\end{proof}

\begin{remark} \label{rem:reader}
In fact, (\ref{prop:ext}$(5)$) and (\ref{lem:zam}$(2b)$) are true under the condition $\cw^G(\ker(f))<\infty$, instead of $\cw(G)<\infty$. The reason is that (\ref{lem:iwa}) can be proved under the assumption that $\cw^G(B)<\infty$, instead of $\cw(G)<\infty$. We leave the verification of this to the interested reader.
\end{remark}

We finish this section by proving that every group from $\W$ is absolutely connected.

\begin{theorem} \label{thm:conn}
Every group from $\W_n$ is $4n$-absolutely connected.
\end{theorem}
\begin{proof} Let $P\subseteq G$ be a thick subset. We will prove that $P^{4n} = G$. By (\ref{cor:4}) there is a thick and normal subset $Q \subseteq P^4$. Since $G\in\W$, we have $Q \cap \Gi_n(G) \neq \emptyset$ (because otherwise $Q\subseteq G \setminus \Gi_n(G)$ and then $G \setminus \Gi_n(G)$ would be thick). Take $g\in Q \cap \Gi_n(G)$. Since $Q$ is symmetric, we have that \[g^G \cup {g^{-1}}^G \subseteq Q^G = Q \subseteq P^4.\] It is enough to use (\ref{def:wsim}).
\end{proof}

Examples of absolutely connected groups are given in Sections \ref{sec:ex} and \ref{sec:other}.

\section{Type-absolute connectedness} \label{sec:00}

In this section we use the notation from Section \ref{sec:abscon}. We introduce the concept of type-absolute connectedness: $G$ is \emph{type-absolutely connected} if ${G^*}^{00}_A = G^*$, for every small subset $A$ of $G$ and any monster model $G^*$ of $G$. Note that we restrict $A$ to be a subset of $G$, not of $G^*$ as in the definition of absolute connectedness. We prove basic properties of type-absolutely connected groups.


\begin{remark} \label{rem:00}
The following conditions are equivalent.
\begin{enumerate}
\item ${G^*} = {G^*}^{00}_A$ holds.
\item There is no family $\{P_n : n\in\N\}$ of $A$-definable thick subsets of $G^*$ such that 
\begin{enumerate}
\item $P_{n+1}P_{n+1}\subseteq P_n$, for each $n\in\N$,
\item $P_0$ is a proper subset of $G^*$.
\end{enumerate}
\end{enumerate}
\end{remark}
\begin{proof} $(1)\Rightarrow(2)$ Suppose $\{P_n : n\in\N\}$ is such a family. Then the intersection $H=\bigcap_{n\in\N}P_n$ is a type-definable over $A$ proper subgroup of $G^*$. Since every $P_n$ is thick, $H$ has bounded index in $G^*$. Therefore ${G^*}^{00}_A\subseteq H\subseteq P_0\neq G^*$. $(2)\Rightarrow(1)$ follows by Lemma \ref{lem:00}.
\end{proof}

\begin{definition} \label{def:typeabscon}
A group $G$ is called \emph{type-absolutely connected} if for any sufficiently saturated extension $G^*$ of any first order expansion $(G,\cdot,\ldots)$ of $G$, we have $G^* = {G^*}^{00}_G$. Equivalently (by (\ref{rem:00})), for an arbitrary family $\{P_n : n\in\N\}$ of thick subsets of $G$ satisfying (\ref{rem:00}$(2a)$), $P_n=G$ holds for every $n\in\N$.
\end{definition}

Note that every type-absolutely connected group is infinite (as $\{e\}$ is thick in any finite group). 
Every absolutely connected group is type-absolutely connected. The converse is not true (cf. Corollary \ref{cor:quasi}$(2)$). In fact, if $G$ is type-absolutely connected but not absolutely connected, then ${G^*}^{00}_{\emptyset}/{G^*}^{\infty}_{\emptyset}$ is nontrivial for some saturated extension $G^*$ of $G$.

We prove in (\ref{prop:typecon}) a stronger version of (\ref{prop:ext}$(3)$) for the class of type-absolutely connected groups. In fact (\ref{prop:typecon}) is false for absolutely connected groups (cf. Corollary \ref{cor:quasi}$(2)$).

\begin{proposition} \label{prop:typecon}
If $G$ is generated by a family $\{G_i\}_{i\in I}$ of type-absolutely connected subgroups, then $G$ itself is type-absolutely connected.
\end{proposition}
\begin{proof} Take an arbitrary family $\{P_n : n\in\N\}$ of thick sets in $G$ satisfying (\ref{rem:00}$(2a)$). Fix $n\in\N$, $i\in I$ and consider $P'_n = G_i\cap P_n$. Clearly $P'_{n+1}P'_{n+1}\subseteq P'_n$, so since $G_i$ is type-absolutely connected, $G_i\subseteq P_n$ for any $i\in I$. Similarly, by (\ref{rem:00}$(2a)$), $P_n$ contains any finite product $G_{i_1}\cdot\ldots\cdot G_{i_m}$. Hence, $P_n=G$.
\end{proof}

The proof of the following proposition is similar to the proof of (\ref{prop:ext}$(1,2)$).

\begin{proposition} \label{prop:tcext}
Let $f\colon G \to H$ be an epimorphism of groups.
\begin{enumerate}
\item If $G$ is type-absolutely connected, then $H$ is also type-absolutely connected.
\item If $H$ and $\ker(f)$ are type-absolutely connected, then $G$ is type-absolutely connected.
\item If $H$ is type-absolutely connected, $\ker(f)$ is finite and $G$ has no subgroups of finite index, then $G$ is type-absolutely connected.
\end{enumerate}
\end{proposition}

\begin{proposition} \label{prop:perf}
Every type-absolutely connected group is perfect.
\end{proposition}
\begin{proof}
First we verify that the abelian groups $(\Z/p\Z,+)$ and 
$\Z(p^{\infty})=\left(\Z\left[\tfrac{1}{p}\right]/\Z,+\right)$, for prime $p$, are not type-absolutely connected. The group $\Z/p\Z$ is not type-absolutely connected, because is finite. Let $H = \Z(p^{\infty})$. Then $H$ can be viewed as a dense subgroup of the circle group $S^1$. Small connected neighborhoods of $e$ in $S^1$ form a collection of thick subsets, contradicting the condition from (\ref{def:typeabscon}). Since $H$ is dense in $S^1$, intersections of these neighborhoods with $H$ also work for $H$. Thus $H$ is not type-absolutely connected.

Suppose, for a contradiction, that there is a non perfect type-absolutely connected group $G$. Then by (\ref{prop:tcext}$(1)$) $G/[G,G]$ is type-absolutely connected. The next lemma (which is presumably already known but we could not find a relevant reference) implies that our group $G/[G,G]$ may be mapped homomorphically onto $\Z/p\Z$ or $\Z(p^{\infty})$, which is impossible.
\end{proof}

\begin{lemma} \label{lem:abel}
Every abelian group can be homomorphically mapped onto $\Z/p\Z$ or $\Z(p^{\infty})$, for some prime $p$.
\end{lemma}
\begin{proof}
Let $G$ be an abelian group. We may assume that $G$ is infinite, because otherwise $G$ can be mapped onto $\Z/p\Z$. We may also assume that $G$ is a torsion group. Indeed, suppose that there is $g\in G$ of infinite order. Let $H < G$ be a subgroup of $G$ maximal subject to being disjoint from the set $\{g^n : n\in\Z\setminus \{0\}\}$. When $G/H \cong \Z$, $G$ can be mapped onto $\Z/2\Z$. If $G/H \not\cong \Z$, then $G/\langle H,g\rangle$ is a nontrivial torsion group (for every $a\in G\setminus H$, there is $n$, such that $g^n \in \langle H,a \rangle$, so $a^m\in \langle H,g \rangle$ for some $m$). Every torsion abelian group $G$ splits as a direct sum of Sylow $p$-subgroups $G = \bigoplus_{p\in \PP}G_p$, where each element of $G_p$ has order $p^n$ for some $n\in\N$. Hence we may assume that $G = G_p$ is a $p$-group. If $G \neq p\cdot G$, then $G/pG$ is a vector space over finite field $\F_p$, so there is a mapping from $G$ onto $\Z/p\Z$. If $G = pG$, then $G$ is divisible.
By a well known fact, every divisible abelian group splits as a direct sum of groups isomorphic to $\Q$ or to $\Z[p^{\infty}]$. This finishes the proof.
\end{proof}

Since every type-absolutely connected group is perfect, the natural question arises about the commutator width of these groups. In general, type-absolutely connected group might have infinite commutator width. For instance by (\ref{rem:shtern}), the topological universal cover $\widetilde{\SL_2(\R)}$ is type-absolutely connected and $\cw\left(\widetilde{\SL_2(\R)}\right)=\infty$. Proposition \ref{prop:cw} below answers this question for groups from $\W$ (cf. (\ref{def:wsim})).

\begin{proposition} \label{prop:cw}
For every $N\in\N$ there is a constant $k_N$ such that every group from $\W_N$ has commutator width at most $k_N$.
\end{proposition}
\begin{proof}
Suppose that for every $k\in\N$ there is $G_k\in\W_k$ with commutator width at least $k$. Take $g_k\in G_k$ of commutator length at least $k$. Consider the product $G = \prod_{k\in\N}G_k$. The group $G$ is in $\W_N$ (Lemma \ref{lem:zam}$(1)$), hence $G$ is perfect. However $G \neq [G,G]$, because the element $g=(g_k)_{k\in\N}\in G$ has infinite commutator length, so is not in $[G,G]$.
\end{proof}

\subsection{Quotients}
We give a general remark about the model-theoretic connected components of a quotient group. We use this remark in (\ref{rem:linquot}). Suppose $f\colon G\to H$ is a surjective homomorphism and $G$, $H$, $f$ are $A$-definable in some structure $\M$. For example, suppose $\M=G$ and $\ker(f)$ is $A$-definable subgroup of $G$. Assume that $\M$ (and so $G$ and $H$) is sufficiently saturated. We describe the relationship between the connected components of $G$ and $H$. It is clear that $f\left(G^{\infty}_A\right) = H^{\infty}_A$, $f\left(G^{00}_A\right) = H^{00}_A$ and $f\left(G^{0}_A\right) = H^{0}_A$. Therefore $f$ induces natural surjective homomorphisms $f^{00}\colon G/G^{00}_A \to H/H^{00}_A$ and $f^{\infty}\colon G/G^{\infty}_A \to H/H^{\infty}_A$.

\begin{remark}  \label{rem:quotient}
If $\ker(f)$ is absolutely connected, then $f^{-1}(H^{\infty}_A)=G^{\infty}_A$ and $G/G^{\infty}_A \cong H/H^{\infty}_A$. If $\ker(f)$ is type-absolutely connected, then $f^{-1}(H^{00}_A)= G^{00}_A$ and $G/G^{00}_A \cong H/H^{00}_A$.
\end{remark}
\begin{proof} We prove the first part. The intersection $\ker(f)\cap G^{\infty}_A$ is $A$-invariant bounded index subgroup of $\ker(f)$, hence $\ker(f)\subseteq G^{\infty}_A$. The proof of the second part is similar.
\end{proof}

\section{Generalized quasimorphisms}

Our goal in this section is to construct a group $G$ such that ${G^*}^{00}_{\emptyset}/{G^*}^{\infty}_{\emptyset}$ is far from being abelian (Theorem \ref{thm:nonab}). To this aim we first introduce and use the concept of a \emph{generalized quasimorphism} of a group (\ref{def:quasim}) and then, using certain method of recovering of a compact Hausdorff group from a dense subgroup via standard part map (\ref{prop:recov}), we construct such an example (\ref{thm:nonab}).

A \emph{quasimorphism} or \emph{pseudocharacter} on a group $G$ is a real-valued function $f\colon G \to \R$ such that for some $r\in\R$, called the \emph{defect of $f$}, the following holds $|f(xy)-f(x)-f(y)|<r$, for all $x,y\in G$ \cite{shtern1,kot}.

If $G$ is type-absolutely connected, then by (\ref{prop:perf}), every homomorphism from $G$ to $(\R,+)$ is trivial. However, $G$ might admit many unbounded quasimorphism. For example, a free product $G=G_1*G_2$ of two absolutely connected groups $G_1$ and $G_2$ is type-absolutely connected (\ref{prop:typecon}). The next example provides many unbounded quasimorphisms on $G$.

\begin{example} \label{ex:quasi}
\cite[5.a)]{shtern1} Suppose $V$ and $W$ are nontrivial groups and decompose $V$ into a disjoint union $V_1 = V_+ \dotcup V_- \dotcup V_0$, where $V_+^{-1}=V_-$ (for example $V_+=\{v\}$, for $v^2\ne e$). Define $s\colon V\to\Z$ as: $s(v)=1$ if $v\in V_+$, $s(v)=-1$ if $v\in V_-$ and $s(v)=0$ if $v\in V_0$. Then the map $f\colon V*W \to \Z$ defined as $f(v_1w_1v_2w_2\ldots)=\Sigma s(v_i)$ is an unbounded quasimorphism of defect 3. Note that $f$ is surjective.
\end{example}

\begin{definition} \label{def:quasim}
We say that $f\colon G\to H$ is a \emph{generalized quasimorphisms} with defect $S\subset H$ if $f(x)f(y)f(xy)^{-1}\in S$, for all $x,y\in G$.
\end{definition}

Below is a variant of Lemma \ref{lem:thick}$(2)$ for generalized quasimorphisms.

\begin{proposition} \label{prop:thickquasi}
Suppose $f\colon G\to H$ is a generalized quasimorphism of defect $S$. Let $P\subseteq H$ be $n$-thick and define $Q=f^{-1}(SP)\cap f^{-1}(SP)^{-1}$. Then
\begin{enumerate}
\item $Q$ is $n$-thick in $G$,
\item if additionally $SP = PS$ and $S=S^{-1}$, then for each $m\in\N$ \[Q^m \subseteq f^{-1}\left(S^{2m-1}P^m\right)\cap f^{-1}\left(S^{2m-1}P^m\right)^{-1}.\]
\end{enumerate}
\end{proposition}
\begin{proof} $(1)$ follows from $f(xy^{-1})\in Sf(x)f(y)^{-1}$. $(2)$ follows from $f^{-1}(A)\cdot f^{-1}(B) \subseteq f^{-1}\left(S^{-1}AB\right)$ for $A,B\subseteq H$.
\end{proof}

\begin{corollary} \label{cor:quasi}
\begin{enumerate}
\item Every quasimorphism on an absolutely connected group is a bounded function.
\item Suppose $G_1$ and $G_2$ are infinite groups. The free product $G=G_1*G_2$ is not absolutely connected. If moreover $G_1$ and $G_2$ are type-absolutely connected, then by (\ref{prop:typecon}), $G$ is type-absolutely connected. Hence $G^*={G^*}^{00}_{\emptyset}\neq {G^*}^{\infty}_{\emptyset}$, for some monster model $G^*\succ G$.
\end{enumerate}
\end{corollary}
\begin{proof}
$(2)$ follows from $(1)$ and Example \ref{ex:quasi}. $(1)$ Let $f$ be a quasimorphism on $G$ of defect $r\in\R$. Suppose that $\im(f)\subseteq \R$ is unbounded. For each $N\geq 1$ there is a thick subset $P_N$ of $(\R,+)$ such that \[\im(f)\not\subseteq P_N^N+\left(-(2N-1)r,(2N-1)r\right).\] For example as $P_N$ one can take $g^{-1}(-\varepsilon,\varepsilon)$, where $g\colon \R\to \R/t\Z$ is the quotient map and some $\varepsilon,t>0$. By (\ref{prop:thickquasi}), the group $G$ is not absolutely connected. Indeed \[Q=f^{-1}(P_N+(-r,r))\cap f^{-1}(P_N+(-r,r))^{-1}\] is thick in $G$ and  $Q^N\subseteq f^{-1}\left(P_N^N+(-(2N-1)r,(2N-1)r)\right)\neq G$.
\end{proof}

The proof of the proposition below is straightforward. It says that a free product of a family of quasimorphisms is a generalized quasimorphism.

\begin{proposition} \label{prop:free_prod}
Suppose $\{G_i\}_{i\in I}$ is a family of groups and $f_i\colon G_i \to \Z = x_i\Z$ is a quasimorphism of defect $r_i$, for $i\in I$. Let $G=*_{i\in I}G_i$. Then \[f=*_{i\in I}f_i \colon G \to \F_I=\langle x_i, i\in I \rangle\] is a generalized quasimorphism of defect $S = \bigcup_{i\in I} \bigcup_{k=-r_i}^{r_i} {x_i^k}^{\F_I}$, where $\F_I$ is a free group of basis $\{x_i\}_{i\in I}$ and ${x_i^k}^{\F_I}$ is the conjugacy class of ${x_i}^k$ in $\F_I$.
\end{proposition}

We prove now that every generalized quasimorphism induces a homomorphism after passing to a saturated extension and taking quotient by the subgroup generated by the defect.

\begin{proposition} \label{prop:homm}
Suppose $f\colon G \to H$ is a surjective generalized quasimorphism of defect $S$, where $S=S^{-1}$ and $S$ is normal in $H$. Assume that $G$, $H$, $f$ and $S$ are $A$-definable in some structure $\M$. Then for a monster model $\M^*$ of $\M$, 
%
%
$f^*$ induces an epimorphism \[\overline{f}\colon \overline{G} \to \overline{H}/\left\langle \overline{S}\right\rangle,\] where $\overline{G}=G^*/{G^*}^{\infty}_A$,  $\overline{H}=H^*/{H^*}^{\infty}_A$ and $\left\langle \overline{S}\right\rangle$ is generated in $\overline{H}$ by $\overline{S}= S^*/{H^*}^{\infty}_A$.
\end{proposition}

The components ${G^*}^{\infty}_A$ and ${H^*}^{\infty}_A$ of $G^*$ and $H^*$ in (\ref{prop:homm}) are calculated with respect to the structure on $G$ and $H$ induced from $\M$.

\begin{proof}
Since $\overline{H}/\left\langle \overline{S}\right\rangle = H^*/\left\langle S^*\right\rangle{H^*}^{\infty}_A$, it is enough to check that $f^*({G^*}^{\infty}_A)\subseteq \left\langle S^*\right\rangle{H^*}^{\infty}_A$. Then it is immediate that $\overline{f}$ is an epimorphism. Since $f(xy)\in S^{-1} f(x)f(y)$ and ${G^*}^{\infty}_A$ is generated by $X_{\Theta_A}$ (cf. Section \ref{sec:modcon}) we get $f^*({G^*}^{\infty}_A) = f^*\left(\left\langle X_{\Theta}\right\rangle\right)\subseteq \left\langle S^*\right\rangle \left\langle f^*\left( X_{\Theta_A}\right)\right\rangle$. By compactness, (\ref{lem:thicks}$(1)$) and (\ref{prop:thickquasi}$(1)$), we have \[f^*(X_{\Theta_A}) = \bigcap\left\{f^*(Q) : Q \underset{A\text{-def. thick}}{\subseteq} G^*\right\} \subseteq \bigcap\left\{S^*P : P \underset{A-\text{def. thick}}{\subseteq} H^*\right\}=S^*X_{\Theta_A},\] so $\left\langle f^*\left( X_{\Theta_A}\right)\right\rangle\subseteq \left\langle S^*\right\rangle{H^*}^{\infty}_A.$
\end{proof}

In (\ref{prop:recov}) below we recover a compact topological group from its dense subgroup.

\begin{proposition} \label{prop:recov}
Suppose $C$ is a compact Hausdorff topological group and $D < C$ is a dense subgroup. Let $\{U_i\}_{i\in I}$ be a family of open subsets of $C$ such that
\begin{enumerate}
\item $\{U_i\}_{i\in I}$ is a basis of the topology at the identity $e$ of $C$, closed under finite intersections,
\item for every $i\in I$, there is $j\in I$ such that $\overline{U_{j}}^2\subseteq U_i$ and $U_i=U_i^{-1}$,
\item $x^{-1}U_ix=U_i$ for $i\in I$, $x\in C$.
\end{enumerate}
Suppose $D$ is equipped with some structure $\HH=(D,\cdot,\ldots)$ in a such a way that for some $A\subseteq D$, $U'_i=U_i\cap D$ is $A$-definable in $\HH$ for each $i\in I$ (for example $\HH = (D,\cdot,U'_i)_{i\in I}$, where each $U'_i$ is a predicate and $A=\emptyset$).

Then for every monster model $\HH^*$, there exists a surjective homomorphism $\st\colon \HH^* \to C$, called the \emph{standard part map}, such that
\begin{enumerate}
\item[(a)] $\ker(\st)=\bigcap_{i\in I}{U'_i}^*$ is an $A$-type-definable bounded index normal subgroup of $\HH^*$,
\item[(b)] $\st(d)=d$, for $d\in D$; for a definable subset $S$ of $D$, if $S^*$ is the interpretation of $S$ in $\HH^*$, then $\st(S^*)=\overline{S}$ (the closure is taken in $C$),
\item[(c)] $\st$ induces a homeomorphic isomorphism $\widetilde{\st} \colon \HH^*/\ker(\st) \to C$, where $\HH^*/\ker(\st)$ is endowed with the logic topology.
\end{enumerate}
\end{proposition}
\begin{proof}
Since $C$ is compact, each $U_i$ is a generic subset of $C$ (cf. (\ref{def:generic})) and by (\ref{lem:generic}$(2)$) and $(2)$, $U_i$ is thick in $C$. Hence $U'_i = U_i\cap D$ is thick in $D$ (cf. (\ref{lem:thick}$(4)$)), and again by (\ref{lem:generic}$(2)$), we obtain that $U'_i$ is generic in $D$.

We define $\st\colon \HH^* \to C$. Let $x\in \HH^*$. For each $i\in I$ there is $d_i\in D$ such that $x\in d_i {U'_i}^*$. Hence, the family $\left\{d_i\overline{U_i}\right\}_{i\in I}$ has the finite intersection property in $C$. By $(2)$ and $(3)$, there is a unique $y\in C$ such that $\bigcap_{i\in I} d_i\overline{U_i} = \{y\}$. Indeed, if $y_1,y_2\in \bigcap_{i\in I} d_i\overline{U_i}$, then $y_2^{-1}y_1\in \bigcap_{i\in I}\overline{U_i}^2 = \{e\}$. Define $\st(x)$ as the unique element from $\bigcap_{i\in I} \overline{d_i U_i}$.

Note the following property, which will be used in the rest of the proof: for $x\in \HH^*$, $d\in D$ and $i\in I$, if $x\in d {U'_i}^*$, then $\st(x)\in d\overline{U_i}$.

Since $D$ is dense in $C$, $\st$ is surjective. Hence, the first part $(b)$ is true.

It follows from $(1)$ that $\st^{-1}(e)=\bigcap_{i\in I}{U'_i}^*$, hence $(a)$ follows.

Let $x,x'\in \HH^*$. We prove that $\st(xx')=\st(x)\st(x')$. It is enough to find for each $i\in I$ an element $d\in D$ such that $\st(xx'),\st(x)\st(x')\in d\overline{U_i}$. Fix $i\in I$. There is $j\in I$ satisfying $(2)$ and $d_j,d'_j\in D$ such that $x\in d_j {U'_j}^*$ and $x'\in d'_j {U'_j}^*$. Then \[xx'\in d_jd'_j {{U'_j}^*}^{d'_j}{U'_j}^*\subseteq d_jd'_j {U'_i}^*,\] by $(3)$. Hence $\st(xx')\in d_jd'_j\overline{U_i}$. Moreover $\st(x)\st(x')\in d_jd'_j {\overline{U_j}}^{d'_j}\overline{U_j}=d_jd'_j\overline{U_j}^2\subseteq d_jd'_j\overline{U_i}$.

We prove $(b)$. For $\subseteq$ suppose $s=\st(s^*)$, for some $s^*\in S^*$. By $(1)$ it is enough to prove that for every $i\in I$, $sU_i\cap S\ne\emptyset$. Take $j\in I$ such that $\overline{U_{j}}^2\subseteq U_i$. Suppose $s^*\in d{U'_j}^*$ for $d\in D$. Then $s\in d\overline{U_j}$ in $C$ and $S^*\cap d{U'_j}^*\ne\emptyset$ in $\HH^*$. Hence, the same is true in $\HH$, that is, there is $s'\in S$, such that $s'\in dU'_j$. Thus $d\in s' U'_j$, so $s\in d\overline{U_j} \subseteq s' U_j\overline{U_j}\subseteq s'U_i$, and then $s'\in sU_i$. $\supseteq$ follows by compactness.

We prove, using $(b)$, that $\st$ induces a continuous isomorphism $\widetilde{\st}\colon \HH^*/\ker(\st) \to C$ of compact Hausdorff groups, i.e. $\widetilde{\st}$ is a homeomorphism. It is immediate (see \cite[Section 2]{pill}) that the following family forms a basis of the logic topology for closed sets of $\HH^*/\ker(\st)$: \[\left\{Y_{\varphi}/\ker(\st) : \varphi\text{ formula over }D\right\},\] where $Y_{\varphi}=\{x\in D^* : x\ker(st)\cap\varphi(D^*)\ne\emptyset\}=\varphi(D^*)\ker(st)$. Hence using $(b)$ we get $\widetilde{\st}\left(Y_{\varphi}/\ker(\st)\right)=\st(Y_{\varphi})=\st(\varphi(D^*)) =\overline{\varphi(D)}$ is closed in $C$.
\end{proof}

\begin{remark} \label{rem:basis}
By \cite[1.12]{hofmor}, every compact Hausdorff topological group has \emph{small normal neighbourhoods}, that is, the identity element has a neighbourhood basis consisting of sets invariant under inner automorphisms. Hence, every such group has a basis at the identity satisfying $(1)$, $(2)$ and $(3)$ from Proposition \ref{prop:recov}.
\end{remark}

\begin{theorem} \label{thm:struc}
Suppose $C$ is a compact Hausdorff group and $D<C$ is a dense subgroup generated by $D=\langle d_i \rangle_{i\in I}$. Let $S = \bigcup_{k=-3}^3\bigcup_{i\in I}{d_i^k}^{D}$. Then there exists a type-absolutely connected group $G$ with some first order structure, and an epimorphism \[{G^*}^{00}_{\emptyset}/{G^*}^{\infty}_{\emptyset} \to C/\left\langle \overline{S}\right\rangle,\] where $\overline{S}$ is the closure of $S$ in $C$.
\end{theorem}
\begin{proof}
Let $\{V_i,W_i\}_{i\in I}$ be an arbitrary family of type-absolutely connected groups. For each $i\in I$ take a quasimorphism $f_i\colon V_i*W_i \to \Z$ of defect 3 from Example \ref{ex:quasi}. Let $G=*_{i\in I}V_i*W_i$ and $f'=*_{i\in I}f_i$. Then by Proposition \ref{prop:free_prod}, $f'\colon G \to \F_I$ is a generalized quasimorphism. Compose $f'$ with the natural epimorphism $i\colon \F_I \to D$, to get a generalized quasimorphism $f=f'\circ i\colon G \to D$ of defect $S$. Consider the following structure \[\M = ((G,\cdot),(D,\cdot,S,U'_j)_{j\in J},f),\] where $U'_j=U_j\cap D$, for some basis $\{U_j\}_{j\in J}$ of $C$ satisfying $(1),(2)$ and $(3)$ from (\ref{prop:recov}) (cf. (\ref{rem:basis})). By (\ref{prop:typecon}), $G$ is type-absolutely connected, so by (\ref{prop:homm}) $f$ induces an epimorphism \[\overline{f}\colon {G^*}^{00}_{\emptyset}/{G^*}^{\infty}_{\emptyset} \longrightarrow \overline{D}/\left\langle \overline{S}\right\rangle,\] where $\overline{D}={D^*}/{D^*}^{\infty}_{\emptyset}$. By (\ref{prop:recov}) there is a continuous epimorphism $\widetilde{\st}\colon \overline{D}={D^*}/{D^*}^{\infty}_{\emptyset}\to {D^*}/\ker(\st) \cong C$, such that $\widetilde{\st}(\overline{S})$ is the closure of $S$ in $C$. By composing $\widetilde{\st}\mod \left\langle \overline{S}\right\rangle$ with $\overline{f}$ we get the desired epimorphism.
\end{proof}

We give some applications of (\ref{thm:struc}). 

\begin{example}
Let $D=\langle d_1,d_2 \rangle$ be a free group, freely generated by $d_1,d_2$, let $C=\widehat{D}$ be a profinite completion of $D$ and $S\subset D$ be defined as in (\ref{thm:struc}). Then $\left\langle \overline{S}\right\rangle = N(x_1,x_2)$ is the normal closure of $d_1,d_2$ in $C$ (as every conjugacy class in $\widehat{D}$ is closed). Hence, by (\ref{thm:struc}), $\widehat{D}/N(x_1,x_2)$ is a homomorphic image of ${G^*}^{00}_{\emptyset}/{G^*}^{\infty}_{\emptyset}$, for some type-absolutely connected group $G$. By a recent result \cite[1.6]{nik-segal}, the quotient $\widehat{D}/N(x_1,x_2)$ is abelian.
\end{example}

We now find $C$, $D$ and $S$ as in (\ref{thm:struc}), such that $C/\left\langle \overline{S}\right\rangle$ is non abelian.

The \emph{weight} $w(G)$ of a topological space $G$ is the minimal cardinal number that a base for the topology on $G$ can have.

\begin{theorem} \label{thm:nonab}
Suppose $\{G_n\}_{n\in\N}$ is any family of compact connected Hausdorff topological groups such that each $G_n$ has weight at most $2^{\aleph_0}$ (for example, every compact connected Lie group is such). Then there exists a type-absolutely connected group $G$ such that $\prod_{n\in \N}G_n/ \bigoplus_{n\in\N} G_n$ is a homomorphic image of ${G^*}^{00}_{\emptyset}/ {G^*}^{\infty}_{\emptyset}$.

If infinitely many of the $G_n$'s are non abelian, then ${G^*}^{00}_{\emptyset}/ {G^*}^{\infty}_{\emptyset}$ is also non abelian.
\end{theorem}

Every ultraproduct $\prod_{n\in\N}G_n/\U$ is a homomorphic image of $\prod_{n\in \N}G_n/ \bigoplus_{n\in\N} G_n$. Therefore (\ref{thm:nonab}) implies that there exists a type-absolutely connected $G$ such that ${G^*}^{00}_{\emptyset}/ {G^*}^{\infty}_{\emptyset}$ is non solvable.

\begin{proof}
Let $C=\prod_{n\in \N}G_n$. We find a dense subgroup $D<C$ and $S$ witnessing the assumption of Theorem \ref{thm:struc} and such that $C/\left\langle \overline{S}\right\rangle$ maps homomorphically onto $C/D_1$ where $D_1 =\bigoplus_{n\in\N}G_n$. By \cite[Theorem 4.13]{hofmorw}, there are $x_n,y_n\in G_n$ such that $\langle x_n,y_n\rangle$ is a dense subgroup of $G_n$. Consider \[K = \{e,x_n^k,y_n^k : n\in\N, k=-3,\ldots,3\}\] as a subset of $C$ (under the natural embedding of $G_n$ into $C$). Since $D_1$ is dense in $C$, the subgroup $D=\langle K \rangle$ is also dense in $C$. Note that $K$ is a compact subset of $C$, as $e$ is a limit point of $K$ and $K\setminus\{e\}$ is discrete. The set $S = \bigcup_{k=-3}^3\bigcup_{n\in \N}{x_n^k}^{D}\cup {y_n^k}^{D}$ is contained in $S' = \bigcup_{k=-3}^3\bigcup_{n\in \N}{x_n^k}^{C}\cup {y_n^k}^{C}$. Since $K$ is compact, $S'$ is also compact. Indeed, consider $F\colon C\times C \to C$, $F(a,b)=a^b$. As $F$ is continuous, the image $F(K\times C)=S'$ is compact. Hence $\overline{S}\subseteq S'\subset D_1$, so $C/\left\langle \overline{S}\right\rangle$ maps homomorphically onto $C/D_1 = \prod_{n\in \N}G_n/\bigoplus_{n\in\N}G_n$. The theorem follows by (\ref{thm:struc}).
\end{proof}

\section{Radicals and compactifications} \label{sec:rad}

In this section we introduce the notion of the type-absolutely connected radical of a group (\ref{def:tcrad}). Our goals are to use that notion together with the concept of the von Neumann radical of a group and (\ref{prop:recov}) to
\begin{itemize}
\item classify all compactifications of a topological group $G$ as mappings $G\to G^*/H'$ for some type-definable bounded index subgroups $H'$ of $G^*$ (\ref{thm:compact}), (\ref{rem:compact});
\item prove that the class of type-absolutely connected groups coincides with the class of discretely topologized groups with trivial Bohr compactification (\ref{thm:tcvN});
\item characterize all connected Lie groups which are type-absolutely connected (\ref{rem:shtern});
\item prove certain result about solvable extensions (\ref{thm:typec}).
\end{itemize}

In 1976 Hirschfeld described \cite{hir} all compactifications of a given group $G$ as quotients of an \emph{enlargement} $^*G$ (in the sense of non-standard analysis) of $G$. We obtain a similar result (\ref{thm:compact}) as a corollary of (\ref{prop:recov}). In our case the quotients have the logic topology. We also give a certain description of a group $H$ such that $G^*/H$ is the universal (the Bohr) compactification of $G$. A. Pillay (independently) proved that such $H$ can be characterized in another way.

\begin{definition} \label{def:tcrad}
For a group $G$, by $\RR_{tc}(G)$ we denote the maximal type-absolutely connected subgroup of $G$, that is, a subgroup which itself is type-absolutely connected, and contains all type-absolutely connected subgroups of $G$. We call $\RR_{tc}(G)$ the \emph{type-absolutely connected radical of $G$}. Existence of $\RR_{tc}(G)$ is given by (\ref{prop:typecon}).
\end{definition}

$\RR_{tc}(G)$ is a characteristic subgroup of $G$ and is contained (by (\ref{prop:perf})) in every element of the derived series of $G$, hence is contained in the perfect core of $G$ (the largest perfect subgroup of $G$). In Section \ref{sec:ex} (Theorem \ref{thm:che2}) we prove that for every Chevalley group $G$, the subgroup $\RR_{tc}(G)$ is definable and coincides with $[G,G]$.

\begin{remark} \label{rem:rtc}
$\RR_{tc}$ has the following properties:
\begin{enumerate}
\item $\RR_{tc}(\RR_{tc}(G))=\RR_{tc}(G)$,
\item $\RR_{tc}(G/\RR_{tc}(G))=\{e\}$ (cf. (\ref{prop:tcext}$(2)$)),
\item $\RR_{tc}(G\times H) = \RR_{tc}(G)\times\RR_{tc}(H)$,
\item if $f\colon G\to H$ is a homomorphism, then $f(\RR_{tc}(G))\subseteq \RR_{tc}\left(f(G)\right)$ (by (\ref{prop:tcext}$(1)$)),
\item for an arbitrary structure on $G$ we have $\RR_{tc}(G)\subseteq G \cap {G^*}^{00}_{\emptyset}$.
\end{enumerate}
\end{remark}

We make below the definition of ``absolutely connected subset of a group'' (cf. (\ref{def:abscon}$(1)$)) for use in (\ref{ex:vneu}).

\begin{definition} \label{def:thickset}
A subset $X$ of a group $G$ is called a \emph{$n$-absolutely connected} if for every thick $P\subseteq G$, we have $X\subseteq P^n$.
\end{definition}

\begin{remark}
Note that if $X\subseteq G$ is an absolutely connected subset of $G$ which is also $A$-definable with respect to some structure on $G$, then $X^*\subseteq G\cap {G^*}^{\infty}_A$.
\end{remark}

\begin{example} \label{ex:vneu}
We give an example of a group $G$ for which $\RR_{tc}(G)$ is trivial, but for any structure on $G$, $G \cap {G^*}^{00}_{\emptyset}$ is nontrivial, i.e. where the inclusion in (\ref{rem:rtc}$(5)$) is strict. Suppose $k$ is an arbitrary field of characteristic $0$. Consider \[G = \left\{
\left( \begin{array}{cc}
a & b \\
0 & 1 \\ \end{array}\right)
: a,b\in k, a\ne 0 \right\} \text{ and } H = \left\{ 
\left( \begin{array}{cc}
1 & b \\
0 & 1 \\ \end{array}\right)
: b\in k \right\}.\] We prove that $H$ is a 4-absolutely connected subset of $G$ (but $H$ itself is not an absolutely connected group, as $H$ is abelian, see (\ref{prop:perf})). Suppose $P\subseteq G$ is thick. By (\ref{cor:4}), we may assume that $P^4$ is a normal subset of $G$. Let $h\in H$. Since $\ch(k)=0$, thickness of $P$ implies that for some $n\in \N$, $e\ne h^n\in P^4$. However $h$ and $h^n$ are conjugate in $G$, as \[\left( \begin{array}{cc}
n & 0 \\
0 & 1 \\ \end{array}\right) h \left( \begin{array}{cc}
n & 0 \\
0 & 1 \\ \end{array}\right)^{-1} = h^n.\] Therefore $h\in P^4$. Since $G$ is solvable, $\RR_{tc}(G)$ is trivial. However, if $H$ is $\emptyset$-definable in $G$, then $H\subseteq G\cap {G^*}^{\infty}_{\emptyset}\subseteq G\cap {G^*}^{00}_{\emptyset}$.
\end{example}

We recall the notion of the \emph{von Neumann radical} (also called the \emph{von Neumann kernel}) and \emph{Bohr compactification}. Let $G$ be a topological group. There exists a ``largest'' compact Hausdorff group $bG$ and a continuous homomorphism $b\colon G\to bG$ such that 
\begin{itemize}
\item the image $b(G)$ is dense in $bG$, and
\item for every continuous homomorphism of dense image $f\colon G\to \widetilde{G}$, where $\widetilde{G}$ is a compact Hausdorff topological group, there is a unique continuous surjective homomorphism $\widetilde{f}\colon bG \to \widetilde{G}$ such that $f = \widetilde{f}\circ b$.
\end{itemize}
The group $bG$ is called the \emph{Bohr compactification} of $G$. The kernel of $b$ is called the \emph{von Neumann radical} of $G$, and is denoted by $\RR_{vN}(G)$. Note that $\RR_{vN}(G)$ depends on the topology on $G$ and is a closed normal subgroup. $G$ is said to be \emph{minimally almost periodic} if $\RR_{vN}(G)=G$ \cite{vnou}, which means that $G$ has trivial Bohr compactification.

Suppose now that $G$ is a group with some first order structure, endowed with the discrete topology. Fix $A\subseteq G$. The subgroup $G\cap {G^*}^{00}_A$ of $G$ is the kernel of the mapping $\iota\colon G\to G^*/{G^*}^{00}_A$. The group $G^*/{G^*}^{00}_A$ with the logic topology is compact Hausdorff and one can easily check that $\iota(G)$ is dense in $G^*/{G^*}^{00}_A$ (as $A\subseteq G$). Therefore $\left(G^*/{G^*}^{00}_A,\iota\right)$ is a \emph{compactification} of $G$.

The next proposition is a consequence of (\ref{prop:recov}). For every Hausdorff topological group $G$ equipped with some structure, we give a sufficient condition to recover the Bohr compactification of $G$, and also all other compactifications of $G$, as $G\to G^*/H$, where $H$ is type-definable bounded index subgroup of $G$.

\begin{theorem} \label{thm:compact}
Suppose $(G,\cdot,\ldots)$ is a topological Hausdorff group equipped with some first order structure $(\ldots)$. Let $\{U_i\}_{i\in I}$ be a basis at the identity of $bG$ satisfying $(1)$, $(2)$ and $(3)$ from Proposition \ref{prop:recov}. Assume that the following subsets of $G$ are definable in $G$:
\begin{itemize}
\item $\RR_{vN}(G)$ (the von Neumann radical),
\item $V_i = b^{-1}(U_i)$, for $i\in I$.
\end{itemize}
Then $G\to G^*/H$ is the Bohr compactification of $G$, where $H=\bigcap_{i\in I}V_i^*$. Furthermore, any compactification of $G$ is of the form $G\to G^*/H'$, for some type-definable over $G$ normal subgroup $H'$ of $G^*$ containing $H$.
\end{theorem}
\begin{proof}
Since $\RR_{vN}(G)$ is definable in $G$, the map $b\colon G\to b(G)$ is interpretable in $G$ (we identify $b(G)$ with $G/\RR_{vN}(G)$). Therefore working in $G^*$ one can find $b^*\colon G^*\rightarrow b(G)^*$, which extends $b$, and also $b(G)^*=G^*/\RR_{vN}(G)^*$. Note that
\begin{itemize}
\item the object $b(G)^*$ with the induced structure from $G^*$ is a sufficiently saturated extension of $b(G)$ with the induced structure,
\item the natural surjective map $G^*/H \to b(G)^*/\bigcap_{i\in I}b(V_i)^*$ is continuous with respect to the logic topology on both groups.
\end{itemize}

Since $b(G)$ is a dense subgroup of $bG$, by (\ref{prop:recov}), there is a surjective homomorphism $\st\colon b(G)^* \to bG$, satisfying $(a)$, $(b)$ and $(c)$ from (\ref{prop:recov}), that is, $\ker(\st)=\bigcap_{i\in I}b(V_i)^*$. Hence, we have a surjective map $\st\circ b^*\colon G^* \to bG$ with $\ker(\st\circ b^*)=H$ and such that $G^*/H \underset{homeo}{\cong} b(G)^*/\ker(\st) \underset{homeo}{\cong} bG$.

The furthermore part follows from the first part and the definition of the Bohr compactification.
\end{proof}

\begin{remark} \label{rem:compact}
As an immediate corollary of (\ref{thm:compact}) we obtain: under the notation of (\ref{thm:compact}) endow $G$ with the following structure \[(G,\cdot,\RR_{vN}(G),V_i)_{i\in I},\] which is an expansion of $G$ by some family of open sets (as $\RR_{vN}(G)$ is closed), then $G\to bG$ is homeomorphically isomorphic to $G\to G^*/H$, where $H$ is defined as in (\ref{thm:compact}). Therefore, as has been also observed by A. Pillay, $H$ might be seen as a kind of \emph{canonical} subgroup of $G^*$. The natural question arises: is $H$ concides with ${G^*}^{00}_G$? We plan to deal with this question in a forthcoming paper. If we additionally assume that $G$ is compact and  that every definable subset of $G$ is Haar measurable, then one can apply argument from \cite[p. 579]{NIP}, to get $H={G^*}^{00}_G$. Because ${G^*}^{00}_G=\bigcap_{j\in J}W_j$ where each $W_j$ is thick, so is generic (\ref{lem:generic}). Hence each $W_j$ has positive measure, so $W_j^2$ has non-empty interior and contains some $V_i$.
\end{remark}

We make some comments on groups with discrete topology.

\begin{proposition} \label{prop:vnrad}
Let $G$ be an arbitrary group with discrete topology.
\begin{enumerate}
\item For every sufficiently saturated extension $G^*$ of $G$, $\RR_{vN}(G)\subseteq G\cap {G^*}^{00}_G$.
\item There exist a family $\{P_i : i\in I\}$ of thick normal subsets of $G$ satisfying (\textdagger) from (\ref{lem:00}) and such that
$\RR_{vN}(G)=\bigcap_{i\in I}P_i$. Moreover $\RR_{vN}(G)= G\cap {G^*}^{00}_{\emptyset}$ in the structure $(G,\cdot,P_i)_{i\in I}$.
\end{enumerate}
\end{proposition}
\begin{proof}
$(1)$ Since $\iota\colon G\to G^*/{G^*}^{00}_G$ is a compactification of $G$, there is $f\colon bG\to G^*/{G^*}^{00}_G$ such that $\iota= f \circ b$, so $\ker(b)=\RR_{vN}(G)\subseteq \ker(\iota)=G\cap {G^*}^{00}_G$.

$(2)$ Choose by (\ref{rem:basis}) a basis $\{U_i : i\in I\}$ of the neighbourhood of identity of $bG$ satisfying $(1)$, $(2)$ and $(3)$ from (\ref{prop:recov}). It is enough to take $P_i=b^{-1}(U_i)$. We prove the moreover part. Inclusion $\subseteq$ follows from $(1)$. For $\supseteq$ note that $\bigcap_{i\in I}P^*_i$ is a $\emptyset$-type-definable bounded index subgroup of $G^*$, so $\bigcap_{i\in I}P^*_i\supseteq {G^*}^{00}_{\emptyset}$.
\end{proof}

An immediate consequence of (\ref{rem:rtc}$(5)$) and (\ref{prop:vnrad}$(2)$) (or of (\ref{thm:compact})) is the following

\begin{corollary}
For an arbitrary topological group $G$ (not necessarily with discrete topology) $\RR_{tc}(G)\subseteq \RR_{vN}(G)$.
\end{corollary}

One can prove (cf. \cite[$5(\alpha)$]{vnou}) that in Example \ref{ex:vneu}, if $G$ is discretely topologized then $\RR_{vN}(G)=H$. That is, $\RR_{tc}(G)$ might be a proper subgroup of $\RR_{vN}(G)$. However, by the next theorem, this cannot happen for minimally almost periodic groups.

\begin{theorem} \label{thm:tcvN}
Consider $G$ with the discrete topology. Then $\RR_{vN}(G)=G$ if and only if $\RR_{tc}(G)=G$. That is, the class of type-absolutely connected groups coincides with the class of discretely topologized minimally almost periodic groups.
\end{theorem}
\begin{proof}
$\Leftarrow$ holds as $\RR_{tc}(G)\subseteq \RR_{vN}(G)$. $\Rightarrow$ follows by (\ref{prop:vnrad}$(1)$). Indeed, if $\RR_{vN}(G)=G$, then $G\subseteq {G^*}^{00}_G$, for any $G^*\succ G$, so $G$ is type-absolutely connected.
\end{proof}

\begin{remark} \label{rem:shtern}
Shtern in \cite{shtern2} described the von Neumann kernel of a connected Lie groups and characterized discretely topologized minimally almost periodic connected Lie groups \cite[Section 4, Corollary 2]{shtern2}. Therefore, by (\ref{thm:tcvN}), this also gives a characterization of connected Lie groups which are type-absolutely connected. Namely, suppose $G$ is a connected Lie group with Levi decomposition $G=SR$, where $R$ is the solvable radical of $G$ (the largest connected solvable normal subgroup of $G$) and $S$ is a Levi subgroup. Then $\RR_{vN}(G)=G$ with respect to discrete topology on $G$ if and only if $S$ has no nontrivial compact simple factors and $R=[G,R]$. In particular every semisimple connected Lie group without compact factors, for example the \emph{topological universal cover $\widetilde{\SL_2(\R)}$ of $\SL_2(\R)$}, is type-absolutely connected. Another proof of this fact can be derived from \ref{thm:typec}.
\end{remark}

We describe how $\RR_{vN}$ behaves under solvable extensions.

\begin{theorem} \label{thm:typec}
Suppose $f\colon G \to H$ is an epimorphism of groups with discrete topology. If $\ker(f)$ is solvable of derived length $m$, then \[f^{-1}(\RR_{vN}(H))^{(m)}\subseteq \RR_{vN}(G).\] In particular if $\RR_{vN}(H)=H$, then $G^{(m)}\subseteq \RR_{vN}(G)$. If additionally $G$ is perfect, then $\RR_{tc}(G)=G$, so $G$ is type-absolutely connected.
\end{theorem}
\begin{proof}
The second part follows from the first part and (\ref{thm:tcvN}). We prove the first part. By (\ref{prop:vnrad}$(2)$), $\RR_{vN}(G)=\bigcap_{i\in I}P_i$, where each $P_i$ is thick and normal in $G$.
\begin{claim}
For $i\in I$ and $n\in\N$, $f^{-1}(\RR_{vN}(H))^{(n)}\subseteq P_i\cdot \ker(f)^{(n)}$.
\end{claim}
\begin{proof}[Proof of the claim]
We prove by induction on $n$. Take an arbitrary $i\in I$. Since each $f(P_i)$ is thick in $H$ and $\{f(P_i)\}_{i\in I}$ satisfies (\textdagger) from (\ref{lem:00}), $\RR_{vN}(H)\subseteq H\cap {H^*}^{00}_{\emptyset}\subseteq f(P_i)$, for a suitable $H^*$. Thus $f^{-1}(\RR_{vN}(H))\subseteq P_i\cdot\ker(f)$. This proves the case $n=0$. 

Suppose that the claim is true for $n$. By $(*)$ from the proof of (\ref{lem:iwa}) we get \[\left\{[x,y] : x,y\in f^{-1}\left(\RR_{vN}(H)^{(n)}\right)\right\}\subseteq P_i^4\cdot \ker(f)^{(n+1)}.\] Since $\ker(f)^{(n+1)}$ is a normal subgroup of $G$ and $i$ is an arbitrary element of $I$, (\textdagger) from (\ref{lem:00}) implies that $f^{-1}(\RR_{vN}(H))^{(n+1)}\subseteq P_i\cdot \ker(f)^{(n+1)}$.
\end{proof}
By the claim $f^{-1}(\RR_{vN}(H))^{(m)}\subseteq P_i$ for each $i\in I$, so $f^{-1}(\RR_{vN}(H))^{(m)}\subseteq \RR_{vN}(G)$.
\end{proof}

\begin{corollary} \label{cor:slcover}
The topological universal cover $\widetilde{\SL_2(\R)}$ of $\SL_2(\R)$ is type-absolutely connected, but not absolutely connected.
\end{corollary}
\begin{proof}
Consider the following exact sequence \[0\to \Z \to \widetilde{\SL_2(\R)} \to \SL_2(\R)\to 0.\] Since $\widetilde{\SL_2(\R)}$ is perfect \cite{wood} and $\SL_2(\R)$ is 12-absolutely connected (\ref{thm:che}), (\ref{thm:typec}) implies that $\widetilde{\SL_2(\R)}$ is type-absolutely connected. $\widetilde{\SL_2(\R)}$ is not absolutely connected by \cite{conv_pillay} or by \ref{cor:quasi}$(1)$. Namely, $\widetilde{\SL_2(\R)}$ is isomorphic with $(\Z\times \SL_2(\R),\cdot)$, where $\cdot$ is defined by the rule $(b_1,x_1)\cdot(b_2,x_2) = (b_1+b_2+h(x_1,x_2),x_1x_2)$, for some 2-cocycle $h\colon\SL_2(\R)\times\SL_2(\R)\to \{-1,0,1\}$ \cite[Theorem 2]{asai}. Hence there exists an unbounded quasimorphism $f\colon \widetilde{\SL_2(\R)} \to \Z$ of defect 1.
\end{proof}

\section{Split semisimple linear groups} \label{sec:ex}

The aim of the present section is to prove (see Theorem \ref{thm:che2}) that groups of rational points of split semisimple simply connected linear groups (so in particular Chevalley groups) over arbitrary infinite fields are 12-absolutely connected.

For the class of connected Lie groups $\CC_{Lie}$, in particular for linear groups over $\R$, we can characterize the subclass $\CC_{Lie, tc}$ of $\CC_{Lie}$ consisting of type-absolutely connected groups (cf. (\ref{rem:shtern})). One may expect that some groups from $\CC_{Lie, tc}$ are absolutely connected. However it is not true that all groups from $\CC_{Lie, tc}$ are absolutely connected, as for example $\widetilde{\SL_2(\R)}$ is not absolutely connected \cite{conv_pillay}.

In the case of quasi-simple Chevalley groups over arbitrary infinite fields we have the following result. By \cite[Theorem M]{gord}, there exists a constant $d$, such that every quasi-simple Chevalley group $G(k)$ over an arbitrary field $k$ is $d\cdot\rank(G)$-boundedly simple (see the comment after (\ref{def:abscon})). Hence $G(k)$ is in $\W_{d\cdot\rank(G)}$ (\ref{def:wsim}), so by (\ref{thm:conn}) $G(k)$ is $4d\rank(G)$-ac. In fact, we prove that $G(k)$ is in $\W_3$ and so is $12$-absolutely connected.

By a \emph{linear algebraic group} we mean an affine algebraic group (cf. \cite{borel,plat-rap}). We assume throughout that $k$ is an arbitrary infinite field. All linear groups are assumed to be connected.

We say that a linear algebraic group $G$ is a \emph{$k$-group} or that $G$ is \emph{defined over $k$}, if the ideal of polynomials vanishing on $G$ is generated by polynomials over $k$; cf. \cite[AG \textsection 11]{borel}. By $G(k)$ we denote the group of $k$-rational points of a $k$-group $G$. A \emph{Borel subgroup} of $G$ is any maximal connected Zariski closed solvable subgroup of $G$; a subgroup of $G$ is called \emph{parabolic} if it contains some Borel subgroup. A $k$-group $G$ is called
\begin{itemize}
\item \emph{reductive} (resp. \emph{semisimple}), if the unipotent radical $\RR_u(G)$ (resp. the solvable radical $\RR(G)$) of $G$ is trivial,
\item \emph{split over $k$} or \emph{$k$-split} if some maximal torus $T$ in $G$ is split over $k$ \cite[18.6]{borel}, 
\item \emph{absolutely almost simple} (resp. \emph{almost simple over $k$} or \emph{almost $k$-simple}) if $G$ has no proper nontrivial normal connected closed subgroup (resp. defined over $k$) \cite[0.7]{borel-tits}, \cite[1.1.1, 3.1.2]{tits_class},
\item \emph{simply connected} if $G$ does not admit any nontrivial central isogeny $\pi\colon\widetilde{G} \to G$ (\cite[2.1.13]{plat-rap}, \cite[8.1.11]{springer}, \cite[1.5.4, 2.6.2]{tits_class}).
\end{itemize}

Every absolutely almost simple $k$-group is almost $k$-simple, but the converse is not true in general. However, when $G$ is $k$-split, semisimple and simply connected, then these two notions are equivalent. We use this fact in the proof of Theorem \ref{thm:che}.

For every separable field extension $k'\supseteq k$ there is a functor $\Res_{k'/k}$ \cite[6.17]{borel-tits} which assigns, by restriction of scalars, to each affine $k'$-group $H'$ an affine $k$-group $H = \Res_{k'/k}\left(H'\right)$ such that $H'(k') \cong H(k)$.

\begin{lemma} \label{lem:sim}
Every almost $k$-simple $k$-split simply connected semisimple $k$-group $G$ is absolutely almost simple.
\end{lemma}
\begin{proof}
By \cite[6.21 (II)]{borel-tits}, \cite[3.1.2]{tits_class}, there exists a finite separable field extension $k'$ of $k$ and absolutely almost simple and simply connected $k'$-group $G'$ such that $G = \Res_{k'/k}(G')$. It is enough to prove that $k'=k$. Let $T'$ be a maximal $k'$-split torus of $G'$. Since $G$ is $k$-split, by \cite[16.2.7]{springer}, $T = \Res_{k'/k}(T')$ is a maximal $k$-split torus of $G$ and $\dim(T)=\dim(T')$. However, by \cite[6.17]{borel-tits}, $\dim(T)=\dim(T')[k':k]$, so $k'=k$.
\end{proof}

Throughout this section unless otherwise is stated we use the following notation (\cite[18.6, \textsection 20, \textsection 21]{borel}):

\begin{itemize}
\item $G$ is a connected reductive $k$-split $k$-group, 
\item $T$ is a maximal $k$-split torus in $G$,
\item $\Phi = \Phi(T,G)$ is the root system of \emph{relative $k$-roots} of $G$ with respect to $T$ (each $\alpha\in\Phi$ is a homomorphism $\alpha \colon T \to k^{\times}$); in fact $\Phi$ can be regarded as a \emph{root system} in $\R^n$, satisfying the \emph{crystallographic condition}: $<\alpha,\beta> = 2\frac{(\alpha,\beta)}{(\beta,\beta)} \in \Z$, for all $\alpha, \beta \in \Phi$, where $(,)$ is the usual scalar product on $\R^n$ \cite[14.6, 14.7]{borel},
\item $\Pi \subset \Phi$ is the simple root system generating $\Phi$; that is, every root $\alpha\in\Phi$ can be written as a linear combination of roots from $\Pi$, where all non zero coefficients are positive integers or all are negative integers,
\item $U_{\alpha}$ for $\alpha\in \Phi$, is the $k$-root group of $G$ corresponding to $\alpha\in \Phi$,
\item $\Phi^+$ and $\Phi^-$ are the sets of all positive and negative $k$-roots from $\Phi$,
\item $U$ (resp. $U^-$) is the group generated by all $U_{\alpha}$ for $\alpha\in\Phi^+$ (resp. $\alpha\in\Phi^-$).
\end{itemize}

If $\alpha = \sum_{\beta\in\Pi} k_{\beta}\beta$ is the representation of $\alpha\in\Phi$ with respect to $\Pi$, then the \emph{height} of $\alpha$ is $\htt(\alpha) = \sum_{\beta\in\Pi}k_{\beta}$.

For each $\alpha\in \Phi$ the group $U_{\alpha}(k)$ is a connected and unipotent subgroup, normalised by $T(k)$. In particular \cite[18.6]{borel}, there is an isomorphism $x_{\alpha}\colon k \to U_{\alpha}(k)$ such that for $s\in k$ and $t\in T(k)$, \[tx_{\alpha}(s)t^{-1} = x_{\alpha}(\alpha(t)s). \tag{\textasteriskcentered} \label{eqno1}\] 
More precisely, for $u\in k^{\times}$, $\alpha\in \Phi$, define $w_{\alpha}(u) = x_{\alpha}(u) x_{-\alpha}\left(-u^{-1}\right)x_{\alpha}(u)$ and $t_{\alpha}(u) = w_{\alpha}(u) w_{\alpha}(1)^{-1}$. Furthermore, if $G$ is simply connected, then \cite[Lemma 20(c), p. 29]{stein} $T(k) = \left\langle t_{\alpha}(u) : u\in k^{\times},\ \alpha\in \Pi \right\rangle$ and the action of $t_{\alpha}(u)$ on the $k$-root subgroup $U_{\beta}$ is given by the formula \[t_{\alpha}(u)x_{\beta}(s)t_{\alpha}(u)^{-1} = x_{\beta}\left( u^{<\beta,\alpha>}s\right). \tag{\textasteriskcentered \textasteriskcentered} \label{eqno2}\]

If we fix an ordering on $\Phi^{+}$, then for every element $x\in U$ there is a unique tuple $(s_{\alpha})_{\alpha\in\Phi^{+}}\in k^{\left|\Phi^{+}\right|}$, such that $x$ can be uniquely written in the form \[x = \prod_{\alpha\in\Phi^{+}}x_{\alpha}(s_{\alpha}), \tag{\textasteriskcentered \textasteriskcentered \textasteriskcentered} \label{eqno3} \] where the product is taken in a fixed order. The analogous fact is true for negative roots $\Phi^{-}$ and $U^-$.

The next definition is well known for linear algebraic groups \cite[12.2]{borel}.

\begin{definition} \label{def:reg}
An element $t\in T(k)$ is called \emph{regular} if $C_{G(k)}(t) \cap U(k) = \{e\}$, where $C_{G(k)}(t)$ is the centralizer of $t$ in $G(k)$. Equivalently (by (\ref{eqno2}) and (\ref{eqno3})) $t$ is regular if and only if for every root $\beta\in\Phi^{+}$, $\beta(t)\neq 1$.
\end{definition}

Proposition \ref{prop:reg} below is a variant of \cite[Proposition]{ellgord}. In the proof we use the Gauss decomposition in Chevalley groups from \cite{cheg}.

\begin{theorem}{\cite[Gauss decomposition Theorem 2.1]{cheg}}
Suppose $k$ is an arbitrary infinite field and $G$ is an absolutely almost simple and simply connected $k$-split $k$-group. Then for every noncentral $g\in G(k)$ and $t\in T(k)$ there exist $v\in U^{-}(k), u\in U(k)$ and $x\in G(k)$ such that $g^x = v\cdot t\cdot u$.
\end{theorem}
 
\begin{proposition}\label{prop:reg}
Suppose that $G$ is an absolutely almost simple and simply connected $k$-split $k$-group. Then $G(k) = \left(t^{G(k)}\right)^3$ for every regular element $t\in T(k)$.
\end{proposition}
\begin{proof}
Let $C = t^{G(k)}$. Consider the following functions 
\begin{itemize}
\item $\varphi \colon U(k) \to U(k), \ \varphi(u) = [t,u]=t^{-1}u^{-1}tu$,
\item $\psi \colon U^{-}(k) \to U^{-}(k), \ \psi(v) = [v,t^{-1}] = v^{-1}tvt^{-1}$.
\end{itemize}
The functions $\varphi$ and $\psi$ are well defined, because $U(k)$ and $U^{-}(k)$ are normal subgroups of $B(k)$. Since $t$ is regular, $\varphi$ and $\psi$ are injective. In fact, $\varphi$ and $\psi$ are bijections. We prove that $\varphi$ is surjective; the argument for $\psi$ is similar. Consider the central series $U(k)=U_1(k) \rhd U_2(k) \ldots \rhd U_m(k) = \{e\}$ for $U(k)$, where $U_i(k) = \langle x_{\alpha}(s) : \alpha\in\Phi^{+}, \htt(\alpha)\geq i, s\in k\rangle$. By the Chevalley commutator formula \cite[14.5, Remark(2)]{borel}, we have \[U_i(k)/U_{i+1}(k)\subseteq Z\left(U(k)/U_{i+1}(k)\right). \tag{\dag}\] Each factor $U_i(k)/U_{i+1}(k)$ is a finite dimensional vector space over $k$. Namely, by (\ref{eqno3}), \[U_i(k)/U_{i+1}(k) \cong \left\{\sum_{\htt(\alpha) = i}x_{\alpha}(s_{\alpha}) : s_{\alpha} \in k\right\}.\] By (\ref{eqno1}), $\varphi\left(U_i(k)\right)\subseteq U_i(k)$; hence $\varphi$ induces a $k$-linear transformation \[\varphi_{i} \colon U_i(k)/U_{i+1}(k) \to U_i(k)/U_{i+1}(k),\] with the matrix $\diag(1-\alpha(t): \alpha\in\Phi, \htt(\alpha) = i)$. Since $t$ is regular, each $\varphi_{i}$ is a bijection. Now, using (\dag) one can easily prove, by induction on $1\leq i \leq m-1$, that for every $u\in U(k)$, there is $u_i$ such that $u \equiv [t,u_i] \mod U_{i+1}$. Thus $\varphi$ is surjective.

\begin{claim}
$C^2 \supseteq G(k)\setminus Z(G(k))$
\end{claim}
\begin{proof}[Proof of the claim]
Take an arbitrary $g\in G(k)\setminus Z(G(k))$. By \cite[Theorem 2.1]{cheg} \[g^G \cap U^{-}(k)\cdot t^2\cdot U(k) \neq \emptyset.\] Since $\varphi$ and $\psi$ are surjective, for some $g'\in g^G$ there are $v\in U^{-}(k)$, $u\in U(k)$ satisfying $g' = \psi(v)t^2\varphi(u) = v^{-1}tvu^{-1}tu \in C^2$.
\end{proof}

Our conclusion follows from the claim, because if $x\not\in C^2\cdot C$, then $xC^{-1}\subseteq G(k)\setminus C^2 \subseteq Z(G(k))$; this is impossible since $Z(G(k))$ is finite and $C$ is infinite. 
\end{proof}

We prove that simply connected split groups are 12-absolutely connected.

\begin{theorem} \label{thm:che}
Let $k$ be an arbitrary infinite field and $G$ be a $k$-split, semisimple, simply connected,  $k$-group. Then $G(k)$ is in $\W_3$, so $G(k)$ is 12-absolutely connected.
\end{theorem}
\begin{proof}
Without loss of generality we may assume that $G$ is absolutely almost simple. Indeed, by \cite[22.10]{borel}, \cite[3.1.2]{tits_class}, the group $G$ is a direct product $G_1\times\ldots\times G_n$ over $k$, where each $G_i$ is $k$-split, semisimple, simply connected and almost $k$-simple. Thus $G(k)\cong\prod_{1\leq i \leq n}G_i(k)$. By Lemma (\ref{lem:sim}), every $G_i$ is absolutely almost simple, so by Lemma \ref{lem:zam}$(1)$ and Theorem \ref{thm:conn}, it is enough to prove that $G_i(k)$ is in $\W_3$.

By Proposition \ref{prop:reg} and Theorem \ref{thm:conn}, in order to prove the conclusion of the theorem, one needs to show that the set of non regular elements in $T(k)$ is non thick (cf. (\ref{def:thick})). By Definition \ref{def:reg}, it is enough to find, for each $m\in\N$, a sequence $(t_i)_{i<m}$ in $T(k)$ such that $\beta(t_i^{-1}t_j)\neq 1$ for $i<j<m$ and $\beta\in\Phi$.

By (\ref{eqno1}) and (\ref{eqno2}), $\beta\left(t_{\alpha}(s)\right) = s^{<\beta,\alpha>}$. 
Fix some sequence $\left(\lambda_{\alpha}\right)_{\alpha\in\Pi} \subset \Z$ and for $s\in k^{\times}$, define $a(s) = \prod_{\alpha\in\Pi}t_{\alpha}\left(s^{\lambda_{\alpha}}\right)$. Then for $i<j<m$, \[\beta\left(a\left(s^i\right)^{-1}a\left(s^j\right)\right) = s^{(j-i)\sum_{\alpha\in\Pi}\lambda_{\alpha}<\beta,\alpha>}.\] 
Fix $m\in\N$. Since the Cartan matrix $\left(<\alpha,\beta>\right)_{\alpha,\beta\in\Pi}$ of the irreducible root system $\Phi$ is non degenerate (cf. \cite[Section 3.6]{carter}), one can find a sequence $\left(\lambda_{\alpha}\right)_{\alpha\in\Pi}$ of integers such that, for every simple root $\beta\in\Pi$, the sum $\sum_{\alpha\in\Pi}\lambda_{\alpha}<\beta,\alpha>$ is positive. Every root $\beta$ is a $\Z$-linear combination of simple roots with all coefficients positive, or all negative. Moreover $<\cdot,\cdot>$ is additive in the first coordinate. 
Hence for all $\beta\in\Phi$, \[\sum_{\alpha\in\Pi}\lambda_{\alpha}<\beta,\alpha>\neq 0.\] The field $k$ is infinite, so there is $t\in k^{\times}$ such that $\beta\left(a\left(s^i\right)^{-1}a\left(s^j\right)\right)\neq 1$ for every $i<j<m$. Thus $(t_i)_{i<m} = \left(a\left(s^i\right)\right)_{i<m}$ satisfies our requirements.
\end{proof}

\begin{corollary}
Let $\SL_{\infty}(k)$ be the union (the direct limit) of groups $\SL_n(k)$, where for $n<m$, $\SL_n(k)$ is embedded in $\SL_m(k)$ in a natural way, that is, $A \mapsto \begin{pmatrix} A & 0 \\ 0 & I \end{pmatrix}$. Proposition \ref{prop:ext}$(3)$ and Theorem \ref{thm:che} implies that $\SL_{\infty}(k)$ is 12-absolutely connected.
\end{corollary}

Now we consider the case when $G$ is $k$-split but not necessarily simply connected. Our goal is Theorem \ref{thm:che2} below, which is a more general version of Theorem \ref{thm:che}.

Every semisimple $k$-group $G$ has a \emph{universal $k$-covering} $\pi \colon \widetilde{G} \to G$ defined over $k$ (\cite[Proposition 2.10]{plat-rap} \cite[2.6.1]{tits_class}); that is, $\widetilde{G}$ is a simply connected $k$-group and $\pi$ is a central isogeny defined over $k$.

Denote by $G(k)^+$ the canonical normal subgroup of $G(k)$ generated by the $k$-rational points of the unipotent radicals of parabolic subgroups of $G$ defined over $k$ (\cite{tits_class}) \[G(k)^+ = \left\langle \RR_u(P)(k) : P \text{ is a parabolic subgroup defined over }k \right\rangle.\] 

Let $G$ be a semisimple $k$-split $k$-group and $\pi \colon \widetilde{G} \to G$ be a universal $k$-covering of $G$. Then \cite[6.5, 6.6]{borel-tits-hom}
\begin{enumerate}
\item $\widetilde{G}(k)^+ = \widetilde{G}(k)$ (see for example \cite[Lemma 64, p. 183]{stein}),
\item $\pi\left(\widetilde{G}(k)\right) = G(k)^+$ and $G(k)^+$ is the derived subgroup $[G(k),G(k)]$ of $G(k)$ (cf. \cite[6.5]{borel-tits-hom}).
\end{enumerate}

\begin{theorem} \label{thm:che2}
Let $k$ be an arbitrary infinite field and $G$ be a $k$-split, semisimple $k$-group. Then the derived subgroup $[G(k),G(k)]$ is in $\W_3$, so it is 12-absolutely connected. Moreover $[G(k),G(k)]=\RR_{tc}(G(k))$ (cf. Section \ref{sec:rad}) and it is  $\emptyset$-definable subgroup of $G(k)$ in the pure group language.
\end{theorem}
\begin{proof}
Let $\pi \colon \widetilde{G} \to G$ be a universal $k$-covering of $G$. The group $\widetilde{G}$ is simply connected, so by Theorem \ref{thm:che}, $\widetilde{G}(k)$ is in $\W_3$. By Lemma \ref{lem:zam}$(1)$ $\pi\left(\widetilde{G}(k)\right)=G(k)^+=[G(k),G(k)]$ is 12-absolutely connected. Since $G(k)^+\in\W_3$, one can find $g\in \G_3(G(k)^+)$ (cf. (\ref{def:wsim})). Therefore, since $G(k)^+$ is a normal subgroup of $G(k)$, $G(k)^+$ is definable over $g$. However, $[G(k),G(k)]$ is $\emptyset$-invariant, so is definable over $\emptyset$. As $[G(k),G(k)]$ is absolutely connected, $[G(k),G(k)]\subseteq \RR_{tc}(G(k))$, and $\RR_{tc}(G(k))$ is perfect, so $\RR_{tc}(G(k)) \subseteq [\RR_{tc}(G(k)),\RR_{tc}(G(k))] \subseteq [G(k),G(k)]$.
\end{proof}

Theorem \ref{thm:che2} and (\ref{rem:quotient}) imply the following remark.

\begin{remark} \label{rem:linquot}
\begin{enumerate}
\item  Suppose $G$ is a semisimple $k$-split $k$-group and $G(k)$ is equipped with some first order structure. Let $G(k)^{\rm ab}$ be $G(k)/[G(k),G(k)]$ with the structure inherited from $G(k)$ and $G^*$ be any sufficiently saturated extension of $G(k)$. Then by (\ref{rem:quotient}), $G^*/{G^*}^{\infty}_{\emptyset} \cong {G(k)^{\rm ab}}^*/{{G(k)^{\rm ab}}^*}^{\infty}_{\emptyset}$ and $G^*/{G^*}^{00}_{\emptyset} \cong {G(k)^{\rm ab}}^*/{{G(k)^{\rm ab}}^*}^{00}_{\emptyset}$. In particular ${G^*}/{G^*}^{\infty}_{\emptyset}$ is abelian. As a corollary of (\ref{rem:quotient}) we have that for semisimple split groups, model-theoretic connected components can be described in terms of components of the abelianization.
\item Suppose $G$ is a reductive $k$-split $k$-group. Then $[G,G]$ is a semisimple $k$-split $k$-group \cite[2.3, 14.2]{borel}. Hence, by (\ref{thm:che2}) the group \[G_0 = \left[\left[G,G\right]\left(k\right), \left[G,G\right]\left(k\right)\right]\] is a 12-ac subgroup of $G(k)$ and $G_0 = \RR_{tc}(G(k))$ (\ref{def:tcrad}). Furthermore, if $[G,G](k)$ is $\emptyset$-definable in $G(k)$, then $G_0$ is $\emptyset$-definable in $G(k)$ (by (\ref{thm:che2})). In this case $G^*/{G^*}^{\infty}_{\emptyset} \cong {A}^*/{A^*}^{\infty}_{\emptyset}$, where $A = G(k)/G_0$.
\end{enumerate}
\end{remark}

\subsection{Groups over algebraically closed field}

We consider the case when $k=K$ is an algebraically closed field. In this subsection we assume that $G$ is a connected linear algebraic $K$-group.

Any absolutely connected group is perfect (\ref{prop:perf}). In our context this is also a sufficient condition for absolute connectedness.

\begin{proposition} \label{prop:alg} The following conditions are equivalent:
\begin{itemize}
\item[(a)] $G(K)$ is absolutely connected,
\item[(b)] $G(K)$ is perfect.
\end{itemize}
More precisely, if $\cw(G(K))=r$ and the solvable radical $\RR(G(K))$ is of derived length $m$, then $G(K)$ is $12(4r)^m$-absolutely connected.
\end{proposition}

If $G(K)$ is perfect, the commutator width of $G(K)$ is at most $2\dim(G)$ \cite[2.2]{borel}.

\begin{proof}
$(a)\Rightarrow(b)$ follows from Proposition \ref{prop:perf}. For $(b)\Rightarrow(a)$, note that $H = G/\RR(G)$ is a semisimple linear algebraic $K$-group. As $G(K)$ is perfect, $H(K)$ is perfect and $K$-split, because every semisimle linear algebraic group is split. Hence by (\ref{thm:che2}), $H(K)$ is in $\W_3$. It is enough to apply (\ref{lem:zam}$(2b)$) and (\ref{thm:conn}) to $G(K) \to H(K)$ and $H(K)$.
\end{proof}

Below is an immediate consequence of (\ref{prop:alg}) and results from Section \ref{sec:rad}.

\begin{corollary}
The perfect core of $G(K)$ is absolutely connected and equals $\RR_{tc}(G(K))$.
\end{corollary}

Proposition \ref{prop:alg} cannot be generalized to a nonsplit groups in general. For example $\SO_3(\R)$ is simple as an abstract group (in particular perfect), but is not absolutely connected. 


We do not know if the bound $12(4r)^m$ from (\ref{prop:alg}) is asymptotically sharp.

\begin{question}
Is there a constant $c$ such that every perfect linear algebraic group is $c$-absolutely connected?
\end{question}

\section{More examples} \label{sec:other}

In this section we provide more examples of absolutely connected groups. We consider some infinite permutation groups and infinite-dimensional general linear groups.

\subsection{Infinite permutation groups} \label{sec:permut}
Let $\Omega$ be an infinite set. Consider the alternating group of permutations of $\Omega$: \[\alt(\Omega) = \{\sigma \in \sym(\Omega) : \supp(\sigma) \text{ is finite and }\sigma\text{ is even} \}.\]

The converse of Theorem \ref{thm:conn} is not true.

\begin{proposition} \label{prop:prop}
$\alt(\Omega)$ is 8-absolutely connected but is not in $\W$.
\end{proposition}
\begin{proof}
$\alt(\Omega)$ is not in $\W$ because $\Gi_N(\alt(\Omega)) = \emptyset$, for every $N\in\N$. To prove the absolute connectedness, take an arbitrary $n$-thick subset $P\subseteq \alt(\Omega)$. By considering $P^4$ and using (\ref{cor:4}), we may assume that $P$ is normal. Now it is enough to prove that $P^2 = \alt(\Omega)$.

We claim that $P$ contains any finite product of disjoint even cycles. Take $n$ cycles $(x,a_{1},\ldots,a_{m})$, $(x,b_{1},\ldots,b_{m})$, $\ldots$, $(x,z_{1},\ldots,z_{m})$, where $x$ is the only common element of these cycles. Since $P$ is $n$-thick, we may assume that \[(x,a_1,\ldots,a_m)^{-1} \circ (x,b_1,\ldots,b_m) = (x,b_1,\ldots,b_m,a_m,\ldots,a_1) \in P.\] Therefore $P$ contains every even cycle (because $P$ is normal in $\alt(\Omega)$). In the same way we can prove that $P$ contains any finite product of disjoint even cycles. Namely, instead of one cycle $(x,a_{1},\ldots,a_{m})$, it is enough to consider a product $(x_0,a_{01},\ldots,a_{0n_0})\circ(x_1,a_{11},\ldots,a_{1n_1})\circ\ldots\circ (x_k,a_{k1},\ldots,a_{kn_k})$ of disjoint cycles.

The following two cycles $(x,y,a_1,\ldots,a_{2p+1})$ and $(x,y,b_1,\ldots,b_{2q+1})$ are in $P$, so \[(x,y,a_1,\ldots,a_{2p+1})\circ (x,y,b_1,\ldots,b_{2q+1}) = (x,a_1,\ldots,a_{2p+1})\circ(y,b_1,\ldots,b_{2q+1})\] is in $P^2$. Thus $P^2$ contains every product of two disjoint odd cycles. In a similar way, we can prove that every even permutation (which is a disjoint product of even cycles and an even number of odd cycles) is in $P^2$.
\end{proof}

Now we concentrate on the group $\sym(\Omega)$ and its normal subgroups. For a cardinal $\kappa$ define $\sym^{\kappa}(\Omega) = \{\sigma \in \sym(\Omega) : |\supp(\sigma)|<\kappa\}$. By the previous proposition, $\sym^{\aleph_0}(\Omega)$ has an absolutely connected subgroup $\alt(\Omega)$ of index 2.

We use the following result of Bertram, Moran, Droste and G\"obel (\cite{bertram,moran,droste}).
\begin{thm}[{\cite[Theorem p. 282]{droste}}]
If $\sigma, \tau \in \sym(\Omega)$, $|\supp(\tau)| \leq |\supp(\sigma)|$ and $|\supp(\sigma)|$ is infinite, then $\tau \in \left(\sigma^{\sym(\Omega)}\right)^4$.
\end{thm}

\begin{proposition} \label{ex:sym} Let $\kappa$ be an uncountable cardinal. 
\begin{itemize}
\item[(1)] If $\kappa = \lambda^+$ is a successor cardinal, then $\sym^{\kappa}(\Omega)$ is in $\W_4$, so it is 16-absolutely connected.
\item[(2)] If $\kappa$ is a limit cardinal, then $\sym^{\kappa}(\Omega)$ is 16-absolutely connected, but is not in $\W$.
\end{itemize}
\end{proposition}
\begin{proof} The group $\sym^{\kappa}(\Omega)$ is a normal subgroup of $\sym(\Omega)$ and two elements from $\sym^{\kappa}(\Omega)$ are conjugate in $\sym(\Omega)$ if and only if they are conjugate in $\sym^{\kappa}(\Omega)$. Therefore in the case when $\kappa = \lambda^+$ is a successor cardinal, from the above theorem we conclude that the group $\sym^{\kappa}(\Omega)$ is in $\W_4$ (and non-simple as an abstract group), that is, $\Gi_4(\sym^{\kappa}(\Omega)) = \sym^{\kappa}(\Omega) \setminus \sym^{\lambda}(\Omega)$, and $\left[\sym^{\kappa}(\Omega) : \sym^{\lambda}(\Omega)\right]$ is infinite. When $\kappa$ is a limit cardinal, $\sym^{\kappa}(\Omega) = \bigcup_{\lambda < \kappa} \sym^{\lambda}(\Omega)$. In this case $\sym^{\kappa}(\Omega)$ is not in $\W$, as $\Gi_N(\sym^{\kappa}(\Omega)) = \emptyset$. However for a successor cardinal $\lambda<\kappa$, the group $\sym^{\lambda}(\Omega)$ is in $\W_4$. By Theorem \ref{thm:conn}, $\sym^{\lambda}(\Omega)$ is 16-absolutely connected and by Proposition \ref{prop:ext}$(3)$ the same is true for $\sym^{\kappa}(\Omega)$.
\end{proof}

\subsection{Infinite-dimensional general linear group} \label{ex:inf} The following proposition is derived from the result of Tolstykh (\cite{tolst}).

\begin{proposition} \label{prop:tolst}
Assume that $V$ is an infinite-dimensional vector space over a division ring $D$. Then the group $\GL(V)$ of all linear automorphisms of $V$ is in $\W_{32}$, so $\GL(V)$ is 128-absolutely connected.
\end{proposition}
\begin{proof}
Call a subspace $U$ of $V$ \emph{moietous} if $\dim U = \dim V = \codim U$. We say that $\pi \in G$ is a \emph{moietous involution} if there exists a decomposition $V = U \oplus U' \oplus W$ into a direct sum of moietous subspaces, such that $\pi$ on $W$ is the identity and exchanges $U$ and $U'$. Proposition 1.1 from \cite{tolst} says that every moietous involution is in $\Gi_{32}(\GL(V))$. In order to prove that $\GL(V)$ is in $\W_{32}$ simple it is enough to find in $\GL(V)$ an infinite sequence $(g_i)_{i\in\N}$, such that for all $i<j\in\N$ the element $g_i^{-1}g_j$ is a moietous involution (so $g_i^{-1}g_j \in \Gi_{32}(\GL(V))$). Take an infinite decomposition $V =  \bigoplus_{i\in\N} V_i$ of $V$ into moietous subspaces. Let $g_i$ be a moietous involution of $V$ with respect to the following decomposition: \[V = V_{2i} \oplus V_{2i+1} \oplus \bigoplus_{2i,2i+1 \neq j\in\N} V_j,\] such that $g_i(V_{2i}) = V_{2i+1}$. Then clearly $g_i^{-1}g_j$ for $i\neq j$ is also a moietous involution.
\end{proof}

\begin{acknowledgements}
The author would like to express his gratitude to Ludomir Newelski for his guidance and would like to thank also Cong Chen, Alexander Borovik, Antongiulio Fornasiero, Guntram Heinke, Arieh Lev, Anand Pillay, Andrei Rapinchuk, Anastasia Stavrova, Katrin Tent and John Truss for many useful conversations and Dugald Macpherson for suggestions for improving the presentation of the paper.
\end{acknowledgements}

\end{document}